\documentclass[11pt,a4paper]{article}
\usepackage{amsmath,textcomp, graphicx}
\usepackage[T1]{fontenc}
\usepackage{amssymb, mathrsfs}
\usepackage{amsfonts}
\usepackage{theorem}
\usepackage{enumerate}
\usepackage{epstopdf}
\usepackage[scaled=0.92]{helvet}
\usepackage{mathrsfs}
\usepackage{pdfsync}
\usepackage{color}

\theoremheaderfont{\bfseries}
\newtheorem{tma}{Theorem}[section]
\newtheorem{prop}[tma]{Proposition}

\newtheorem{lema}[tma]{Lemma}

\newtheorem{cor}[tma]{Corollary}
\theorembodyfont{\rmfamily}
\newtheorem{obs}[tma]{Remark}

\newenvironment{proof}{\text{\bf Proof: }}{}

\newcommand{\R}{\mathbb{R}}

\newcommand{\C}{\mathbb{C}}

\newcommand{\B}{\mathcal{B}}

\newcommand{\p}{\partial}
\newcommand{\E}{\mathbb{E}}

\newcommand{\w}{\omega}

\renewcommand{\P}{\mathbb{P}}
\newcommand{\beq}{\begin{equation}}
\newcommand{\eeq}{\end{equation}}
\newcommand{\beqn}{\begin{equation*}}
\newcommand{\eeqn}{\end{equation*}}

\begin{document}

\begin{titlepage}
\vskip 1cm

\begin{center}
{\Large\bf Systems of stochastic Poisson equations:}\\
{\Large\bf hitting probabilities}
\medskip

by\\
\vspace{14mm}

\begin{tabular}{l@{\hspace{10mm}}l@{\hspace{10mm}}l}
{\sc Marta Sanz-Sol\'e}$\,^{(a) (\ast)}$  &and&{\sc No\`elia Viles}$\,^{(\star)}$\\
{\small Facultat de Matem\`atiques i Inform\`atica}          &&{\small @DiMedia, Actividades Digital Media}\\
{\small Universitat de Barcelona }  &&{\small Avinguda Diagonal, 477}\\
{\small Gran Via de les Corts Catalanes, 585}                                      &&{\small 08036 Barcelona, Spain}\\
{\small 08007 Barcelona, Spain}        &&  {\small e-mail:noelia.viles@gmail.com } \\           
{\small e-mail: marta.sanz@ub.edu}       &&\\
{\small http://www.ub.edu/plie/Sanz-Sole} &&  \null 
\end{tabular}

\vskip 1cm

\end{center}

\vskip 1cm

\noindent{\em Abstract.} We consider a $d$-dimensional random field $u=(u(x), x\in D)$ that solves a system of elliptic stochastic equations on a bounded domain $D\subset \R^k$, with additive white noise and spatial dimension $k=1,2,3$. Properties of $u$ and its probability law are proved.
For Gaussian solutions, using results from \cite{DSS2010}, we establish upper and lower bounds on hitting probabilities in terms of the Hausdorff measure and Bessel-Riesz capacity, respectively. This relies on precise estimates on the canonical distance of the process or, equivalently, on $L^2$ estimates of increments of the Green function of the Laplace equation.
 \medskip

\noindent{\em Keywords:} Systems of stochastic Poisson equations; hitting probabilities, capacity; Hausdorff measure.
\smallskip

\noindent{\em AMS Subject Classification.}
{\sl Primary:} 60H15, 60G15, 60J45; {\sl Secondary:} 60G60, 60H07.
\vfill

\noindent
\footnotesize
\begin{itemize}
\item[(a)] Corresponding author.
\item[$^{(\ast)}$] Supported by the grants MTM 2012-31192, MTM 2015- 65092-P, from the \textit{Direcci\'on Gene-\break ral de
Investigaci\'on, Ministerio de Educaci\'on y Ciencia}, Spain.
\item[$^{(\star)}$] Supported by the grant MTM 2012-31192  from the \textit{Direcci\'on General de Investigaci\'on, Mi-\break nisterio de Econom\'{\i}a y Competitividad}, Spain. This work was partially done during a postdoctoral period at the University of Barcelona, supported by the {\em Juan de la Cierva} contract JCI-2011- 08840.
\end{itemize}

\end{titlepage}
\newpage


\section{Introduction}
\label{si}
Let $D$ be a bounded domain of $\R^k$, $k=1,2,3$, for which the divergence theorem holds.
Consider the following system of elliptic stochastic partial differential equations,
\begin{align}
\label{Poisseq}
&-\Delta u^i(x)+f^i(u(x))=g^i(x)+\sum_{j=1}^d\sigma_{ij}\dot{W}^j(x),\  x\in D ,\ i=1,\dots, d,\notag\\
 &u_{|\p D} (x)=0,
  \end{align}
  where $\dot W = (\dot{W}^j, j=1,\ldots,d)$ denotes a $d$-dimensional white noise indexed by $x\in\R^k$, $f: \R^d \rightarrow \R^d$, $g: D\rightarrow \R^d$, and $\sigma = (\sigma_{ij})_ {1\le i,j\le d}$ is a non-singular matrix with real-valued entries. 
  
 The main motivation of this paper has been to find upper and lower bounds for the hitting probabilities $\P\{u(I)\cap A\ne \emptyset\}$, $I\subset D$, $A\subset \R^d$, in terms of the Hausdorff measure and the capacity of the set $A$, respectively. This is a fundamental problem in probabilistic potential theory that, in the context of stochastic partial differential equations, has been extensively studied for the stochastic heat and wave equations. We refer the reader to  \cite{DN2004}, \cite{DKN2009}, \cite{DSS2012}, and references herein, for a representative sample of results.
 
 For $d=1$, equations like \eqref{Poisseq} have been first considered in \cite{BP1990} and then in \cite{DM1992}, in relation with the study of the Markov field property of the solution, and in \cite{GM2006}, \cite{MSS2006}, \cite{SST2009}, for numerical approximations, among others.  
 We observe that in \eqref{Poisseq}, the stochastic forcing is an {\em additive} noise. Therefore, in the integral formulation of the system given in \eqref{mildPoisseq}, the stochastic integral term contains a deterministic integrand and defines a Gaussian process. Since there is no time parameter in \eqref{Poisseq}, considering a {\em multiplicative} noise would require a choice of anticipating stochastic integral in \eqref{mildPoisseq}. For example, one could take the Skorohod integral. This would make the objective of this article difficult and rather speculative. 
 
The content of the paper is as follows. In Section \ref{s1}, we prove the existence and uniqueness of a solution to \eqref{Poisseq}, when the function $f$ satisfies a monotonicity condition (see Theorem \ref{thm31}). This is a $d$-dimensional stochastic process indexed by $\bar D$, the closure of the domain $D$, with continuous sample paths and vanishing at the boundary of $D$, a.s. The proof applies standard methods of the theory of nonlinear monotone operators. In order to make the article self-contained, we include the details of the proof. In Section \ref{s2}, we prove some properties of the solution to \eqref{Poisseq}. With the a priori bound proved in Proposition \ref{up}, we prove that the solution lies in $L^p(\Omega;\R^d)$, uniformly in $x\in D$.
Moreover, by using estimates of increments of the $L^2$-norms of the Green function, we prove that the sample paths of the solution are H\"older continuous. Section 3 is devoted to study some aspects of the law of the solution. The integral formulation \eqref{mildPoisseq} suggests that the
law of $u$ is obtained from a Gaussian process by a non adapted shift. By applying Kusuoka's anticipating extension of Girsanov's theorem (see \cite{K1982}) we obtain the absolute continuity of the law of $u$ with respect to a Gaussian measure. As a trivial by-product, for any $x\in D$, the law of $u(x)$ is absolutely continuous with respect to the Lebesgue measure in $\R^d$. For $d=1$, this result has been proved in \cite{DM1992}.

In the remaining of the article, it is assumed that $f\equiv 0$, therefore focusing on Gaussian solutions. For the sake of simplicity, we also assume $g\equiv 0$. Probabilistic potential theory for Gaussian processes has been the object of extensive work. The more recent developments are on anisotropic random fields with the paradigmatic example of fractional Brownian sheets. Solutions to SPDEs, like the stochastic heat equation, belong to this class. In \cite{X2009}, different results relative to sample paths of anisotropic Gaussian random fields are presented, in particular on hitting probabilities, and a exhaustive number of references are given. The Gaussian process obtained from \eqref{mildPoisseq} with $f=g=0$, provides yet another example of random field for which, to the best of our acknowledge, results on hitting probabilities have not been obtained before.

Criteria for hitting probabilities of general random fields have been established in \cite{DSS2010}. When implemented on Gaussian processes, two fundamental ingredients are needed. The first one concerns the canonical distance $\delta$ on values of the Gaussian process at two different points $x, z$, which is required to commensurate with a pseudo-distance on the parameter set of the process (see \eqref{metric} for the definition of $\delta$). The second ingredient is the property of two-point local non determinism (see conditions (C1), (C2) in \cite[Section 2]{BLX2009}, \cite[p. 158]{X2009}). For random fields obtained from solutions to SPDEs, both conditions are intimately connected with upper and lower bounds of increments of $L^2$--norms of the Green function or the fundamental solution (see \cite{DKN2007}, \cite[Secction 4]{DSS2010} for the stochastic heat and wave equations, respectively).
For the Green function of the Laplace's equation, Section \ref{sa} provides the necessary results. We prove that in spatial dimension $k=1, 3$, those norms of increments commesurates with $|x-z|$ and $|x-z|^{1/2}$, respectively, while for $k=2$ there is a gap (see Lemmas \ref{l-sa-2}, \ref{l-sa-4}, \ref{l-sa-3}, respectively). By applying these results, we establish in Section \ref{ss3-3} upper and lower bounds for the hitting probabilities of the Gaussian process defined in \eqref{s3-11} in dimension $k=1,3$, and upper bounds in dimension $k=2$. Our investigations led to the conclusion that the gap in the estimates of Lemma \ref{l-sa-3} implies that for $k=2$ the two-point local non determinism fails to be true, although we do not have a proof of this statement. 

We end this section with some remarks on a possible extension of our results to $f\neq 0$. A natural approach, inspired by \cite{DKN2007}, consists in applying Proposition \ref{p2-s2-00}. By doing so, from hitting probabilities estimates on the Gaussian solution, and moment estimates of
 the random density \eqref{s2-0-10}, upper and lower bounds on hitting probabilities for $u$ could be obtained. Assume that $f(x) = ax+b$. Then, the $\det_2$ factor on the right-hand side of \eqref{s2-0-10} is constant, and the exponential factor involves the random variable $\delta(f(\omega))$ that belongs to a second order Wiener chaos. With a constraint on the size of the constant $a$, moments up to a certain order of the random density do exist, and the above strategy works well. However, we believe that going beyond that particular case would need new ideas.


\section{Existence and uniqueness of solution to the system of elliptic equations}
\label{s1}

This section is devoted to establish the existence and uniqueness of solution to the system of elliptic equations \eqref{Poisseq}.

We begin with recalling the expression of the Green function associated to the Laplace equation in dimensions $k=1,2,3$, that we denote by $G_D^k$.

For $k=1$, we consider $D=(0,b)$, with $b>0$. In this case,
 \beq
 \label{GD1}
G_D^1(x,y)= x\wedge y -\frac{xy}{b}.
 \eeq
 For $k=2,3$, $D$ is an arbitrary domain with regular boundary, and then,
 \beqn
 G^k_D(x,y)=\Gamma^k(|x-y|)-\E_x(\Gamma^k(|B_{\tau}-y|), \; x,\, y\in D.
 \eeqn
 In this formula, $B_\tau$ is the random variable defined by a Brownian motion $B$ that starts from $x$ at time $t=0$, at the first time (denoted by $\tau$) it hits $\p D$, and
\beq
\label{sa-310}
\Gamma^k(z)=C_k \left\{
  \begin{array}{l }
\log|z|,\, k=2,\\
   |z|^{-1},\; k=3,
  \end{array} \right.
 \eeq
with $C_2=  \frac{1}{2\pi}$ and $C_3 = \frac{1}{4\pi}$ (see \cite{GT2001}). In dimensions $k=2,3$, and for $D= B_1(0)$, the unit ball centered at zero, we give in Section \ref{sa} an alternate formula for $G_D^k(x,y)$ (see \eqref{sa-31}.

The inner product in ${L^2(D;\R^d)}$ and its corresponding norm will be denoted by $\langle\cdot, \cdot \rangle$ and $\|\cdot\|$, respectively. 

For its further and frequent use, we quote a well-known result of the $L^2$--norm of $G_D^k$.

\begin{lema}\label{l-sa-1}
For $k=1$, we consider $D=(0,b)$, $b>0$, and for $k=2,3$, $D$ is an arbitrary bounded domain of $\R^k$ with regular boundary.
We have
\beq
\label{sa-1}
\sup_{x\in D}\|G_D^k(x,\cdot)\|< +\infty.
\eeq
\end{lema}
\begin{proof}
Let $k=1$. Explicit computations based on the expression \eqref{GD1} yield
\beq
\label{sa-10}
\int_0^bG_D^1(x,y)^2dy=\frac{x^2(x-b)^2}{3b}.
\eeq
From this, one trivially gets \eqref{sa-1}.

For $k=2, 3$, the result is proved in \cite[Lemma 3.3]{DM1992}.
\hfill{$\square$}
\end{proof}

Following \cite{BP1990} (see also \cite{DM1992}, \cite{GM2006}, \cite{SST2009}), a stochastic process $u=\{u(x),\; x\in D\}$ satisfying
 \begin{align}
 \label{mildPoisseq}
u^i(x) + \int_DG_D^k(x,y)f^i(u(y))dy  &= \int_DG_D^k(x,y)g^i(y)dy\\
& +\sum_{j=1}^d\sigma_{ij}\int_DG_D^k(x,y) {W}^j(dy),\ i=1,\dots, d,
\end{align} 
a.s. for all $x\in D$, is called a {\sl mild} solution of \eqref{Poisseq}.

We define the (Nemytskii type) operator $\operatorname{F}:L^2(D;\R^d) \longrightarrow L^2(D;\R^d)$ by $F(v)(y)= f(v(y))$, for any $v\in L^2(D;\R^d)$, $y\in \R^d$, and introduce the following assumptions.
\begin{enumerate}
\item[\textbf{(C)}]  $\operatorname{F}$ is strongly continuous and bounded.
\item [\textbf{(M)}] {\bf Monotonicity.} There exists $L>0$ such that for any $u,v\in L^2(D;\R^d)$,
\begin{equation*}
\langle u-v, \operatorname{F}(u)-\operatorname{F}(v)\rangle\geq -L\|u-v\|^2.
\end{equation*}
\end{enumerate}

For its further use, we recall a consequence of Poincar\'e's inequality:\\
\textbf{(P)} There exists a constant $a>0$ such that for any $\varphi \in L^2(D;\R^d)$ , 
\beqn
\left\langle \int_D G_D^k(\cdot,y)\varphi(y)dy,\varphi\right\rangle\ge a \left\| \int_D G_D^k(\cdot,y)\varphi(y) dy\right\|^2.
\eeqn

In the sequel, we denote by $\mathcal{C}(\overline{D};\R^d)$ the space of continuous functions on $\overline{D}$ and set 
$\mathcal{S}=\{\w:\; \w\in \mathcal{C}(\overline{D};\R^d),\; \w|_{\p D}=0\}$. The result on existence and uniqueness of solution for \eqref{mildPoisseq} reads as follows.

\begin{tma}\label{thm31}
Assume that
\begin{enumerate}[(i)]
\item $g\in L^2(D;\R^d)$;
\item $\operatorname{F}$ satisfies the properties \textbf{(C)} and (\textbf{M});
\item the constants $L$ and $a$ in (\textbf{M}) and (\textbf{P}) respectively, satisfy $0<L<a$. 
\end{enumerate}
Then, the system of equations \eqref{mildPoisseq} has a unique solution $\{u(x), x\in D\}$ in $\mathcal{S}$.
\end{tma}

\begin{proof}
For any $\w\in\mathcal{S}$, we set 
\begin{equation}\label{defTop}
\operatorname{T}(\w):= \w+\int_DG_D^k(\cdot, y) f(\w(y))dy.
\end{equation}
This defines an operator $\operatorname{T}:\mathcal{S} \longrightarrow \mathcal{S}$. Indeed, the function $x\longrightarrow \int_DG_D^k(x, y) f(\w(y))dy$ vanishes if $x\in \p D$. Fix $x,\bar x\in D$. 
Cauchy-Schwarz inequality yields
\begin{align*}
\left|\int_D[G_D^k(x,y)-G_D^k(\bar{x},y)]f(\w(y))dy\right| &\leq \|G_D^k(x,\cdot)-G_D^k(\bar{x},\cdot)\| \|f(\w)\|\\
&\leq C  \|G_D^k(x,\cdot)-G_D^k(\bar{x},\cdot)\|.
\end{align*}
Appealing to lemmas \ref{l-sa-2}, \ref{l-sa-3}, \ref{l-sa-4} we see that $x\longrightarrow \int_DG_D^k(\cdot, y) f(\w(y))dy$ belongs to $\mathcal{C}(\overline{D};\R^d)$ and in fact, it is H\"older continuous.

Let 
\beqn
b^i(x)=\int_D G_D^k(x,y)g^i(y)dy+\sum_{j=1}^d\sigma_{i,j}\int_DG_D^k(x,y) W^j(dy), \  i=1,\dots, d.
\eeqn
Clearly, $b^i(x)=0$ if $x\in\p D$. 
The process $\left\{\int_DG_D^k(x,y) W^j(dy), x\in D\right\}$ has continuous sample paths, a.s. Indeed this follows from lemmas \ref{l-sa-2}, \ref{l-sa-3}, \ref{l-sa-4}, the hypercontractivity property and Kolmogorov's continuity criterion. The term $\int_D G_D^k(x,y)g^i(y)dy$ defines also a continuous function in $x$. This follows from Cauchy-Schwarz inequality and again, by applying lemmas \ref{l-sa-2}, \ref{l-sa-3}, \ref{l-sa-4}. Consequently, $b(x) = (b^i(x), i=1,\ldots,d)\in \mathcal{S}$.

We will prove that the operator equation $\operatorname{T}\w =b$ has a unique solution for any $b\in \mathcal{S}$, or equivalently that $\operatorname{T}$ is a bijective operator on $\mathcal{S}$.

For this, we start by checking that $\operatorname{T}$ is one-to-one. Fix $u,\, v\in \mathcal{S}$ and assume that $\operatorname{T} u =\operatorname{T} v$. Then 
\beq
\label{equvdiff}
u(x)-v(x)= - \int_D G_D^k(x,y) [f(u(y))-f(v(y))]dy.
\eeq
By taking the inner product in $L^2(D;\R^d)$ with $F(u)-F(v)$ on both sides of this equality, and applying (\textbf{P}), we obtain
\begin{align*}
\langle u-v, F(u)-F(v)\rangle &= - \left\langle \int_D G_D^k(\cdot, y)[f(u(y))-f(v(y))]dy, F(u)-F(v)\right\rangle
\\&\leq -a\left\|\int_D G_D^k(\cdot, y)[f(u(y))-f(v(y))]dy\right\|^2
\\&=-a\|u-v\|^2.
\end{align*}
 On the other hand, using the property $(\textbf{M})$, we have 
\beqn
\langle u-v, F(u)-F(v)\rangle\geq -L\|u-v\|^2.
\eeqn
Thus, 
\begin{equation*}
-L \|u-v\|^2\leq -a \|u-v\|^2,
\end{equation*}
or equivalently, 
\beqn
(a-L)\|u-v\|^2\leq 0.
\eeqn
Since $L<a$, this implies that $u=v$ in $\mathcal{S}$.

\medskip

Next, we prove that $\operatorname{T}$ is onto, proceeding in a way similar as in \cite{MSS2006}.

\noindent{\sl Step 1: A solution for a regular problem}. For a fixed $b\in \mathcal{S}$, we consider a sequence $(b_n)_{n\geq 1}\in \mathcal{C}_0^{\infty}(D;\R^d)$, such that $b_n\rightarrow b$ in $L^2(D;\R^d)$. Define $\operatorname{A} = -\Delta + F$ restricted to $W_0^{1,2}(D;\R^d)$. We will prove that 
for each $n\ge 1$, there exists $u_n\in\mathcal{S}$ such that $\operatorname{A} u^n=b^n$. 

We remind the classical result on solutions of nonlinear monotone operator equations (see, e.g. \cite[Theorem 26.A, page 557]{Z1990}): 

\noindent{\sl Let $X$ be a reflexive Banach space; denote by $X^*$ its topological dual. Let $B: X\rightarrow X^*$ be a strictly monotone, coercive, hemicontinuous operator. Then, for any $k\in X^*$, the equation $Bw=k$ has a unique solution on $X$.}

This theorem will be applied to $B:= A$ and $X:= W_0^{1,2}(D;\R^d)$. Notice that $\operatorname{A}=-\Delta + F$ coincides with the operator $\operatorname{T}$ on the space $W_0^{1,2}(D;\R^d)\cap\mathcal{S}$. Moreover, for any  $u, v\in W_0^{1,2}(D)$, 
\beqn
\langle \operatorname{A}u, v \rangle =\int_D\nabla u(x)\nabla v(x) dx +\int_D f(u(y))v(y)dy,
\eeqn
or coordinatewise,
\beqn
\langle \operatorname{A}u^i, v^i \rangle =\int_D \nabla u^i(x) \nabla v^i(x) dx +  \int_D f^i(u(y)) v^i(y) dy.
\eeqn
From Poincar\'e's inequality we deduce that for any $u\in W_0^{1,2}(D;\R^d)$, $\Vert \nabla u\Vert_{L^2(D; \R^d)}^2 \ge a \Vert u\Vert^2_{L^2(D; \R^d)}$ (see e.g. \cite[Theorem 7.10, page 155]{GT2001}).
From this inequality and (\textbf{M}) it follows that $\operatorname{A}$ is strictly monotone. Indeed, for any $u,v\in W_0^{1,2}(D;\R^d)$, $u\ne v$, we have
\begin{align*}
\langle \operatorname{A}u-\operatorname{A}v, u-v \rangle &=\|\nabla (u-v)\|^2+\langle F(u)-F(v),u-v\rangle
\\&\geq \|\nabla (u-v)\|^2-L\|u-v\|^2
\\&\geq a\|u-v\|^2-L\|u-v\|^2
\\&=(a-L)\|u-v\|^2>0.
\end{align*}
Using again Poincare's inequality, we have that 
\begin{align*}
\langle \operatorname{A}u, u\rangle &=\|\nabla u\|^2+\langle F(u), u \rangle
\\&= \|\nabla u \|^2+\langle F(u)-F(0), u \rangle+\langle F(0), u \rangle
\\&\geq  \|\nabla u \|^2-L\|u\|^2+\langle F(0), u \rangle
\\&\geq  (a-L)\|u\|^2+\langle F(0), u \rangle.
\end{align*}
Then, since $\left|\frac{\langle F(0), u\rangle}{\|u\|}\right|\leq |F(0)|_{\R^d}$, we see that
\begin{equation*}
\lim_{{\|u\|_{_{W_0^{1,2}(D;\R^d)}}\to +\infty}}\frac{\langle Au, u\rangle}{\|u\|}=+\infty,
\end{equation*}
proving coercivity.

For any $u,v,w \in W_0^{1,2}(D;\R^d)$, we have 
\begin{align*}
\langle A(u+tv), w\rangle &=\int_D\nabla u(x)\nabla w(x) dx+t\int_D \nabla v(x)\nabla w(x) dx
\\&+ \int_D f(u(x)+t v(x))w(x) dx.
\end{align*}
This yields the continuity of the mapping $t\mapsto \langle \operatorname{A}(u+tv), w\rangle$ on $[0,1]$. Thus, $\operatorname{A}$ is hemicontinuous.

Therefore the equation $\operatorname{A} \w = b_n$ has a unique solution on $W_0^{1,2}(D;\R^d)$ that we denote by $u^n$, and the sequence $(u^n)_{n\ge 1}$ satisfies
\beq
\label{unexpr}
u^{n}(x) + \int_D G_D^k(x,y)f(u^{n}(y))dy = b_n(x), 
\eeq
for $x\in D$, and $u^n|_{\p D}=0$.
\smallskip

\noindent{\sl Step 2: Passing to the limit}. We prove that $(u_n)_{n\ge1}$ is a Cauchy sequence in $L^2(D;\R^d)$. Indeed, fix $n,\, m\geq 1$. Starting with the identity
\beqn
u^{n}(x)-u^{m}(x) + \int_D G_D^k(x,y)[f(u^n(y))-f(u^m(y))]dy = b_n-b_m, 
\eeqn
and taking the inner product with $F(u^n)-F(u^m)$ in $L^2(D;\R^d)$, yields
\begin{align*}
&\langle u^n-u^m, F(u^n)-F(u^m)\rangle\\ 
&\quad+ \left\langle \int_D  G_D^k(\cdot, y) (f(u^n(y))-f(u^m(y))) dy, F(u^n)-F(u^m)\right\rangle\\
& \quad= \left\langle F(u^n)-F(u^m), b_n-b_m\right\rangle.
\end{align*}
The assumption (\textbf{M}) and the property (\textbf{P}) implies
\begin{align}
\label{eqdifunum}
-L\|u^n-u^m\|^2&+ a\left\|\int_D  G_D^k(\cdot, y) (f(u^n(y))-f(u^m(y))) dy\right\|^2\nonumber
\\& \le \left\langle F(u^n)-F(u^m), b_n-b_m\right\rangle.
\end{align}
By substracting the expresion (\ref{unexpr}) for $n$ and $m$, respectively, we obtain,
\begin{align*}
\left\|\int_D  G(\cdot, y) (f(u^n(y))-f(u^m(y))) dy\right\|^2
&=\left\|u^n-u^m\right\|^2+\left\|b_n-b_m\right\|^2\\
& -2\langle u^n-u^m,b_n-b_m\rangle.
\end{align*}
Multiplying this identity by $a$ and using (\ref{eqdifunum}), we have 
\begin{align*}
&(a-L)\|u^n-u^m\|^2+ a\left\|b_n-b_m\right\|^2 \\
&\qquad \le \left\langle F(u^n)-F(u^m)+2a(u^n-u^m), b_n-b_m\right\rangle\\
&\qquad \leq \|b_n-b_m\|(\| F(u^n)-F(u^m)\|+2a\|u^n-u^m\|).
\end{align*}
We are assuming $0<L<a$. Hence we conclude that $\lim_{n,m\to\infty}\Vert u^n-u^m\Vert=0$.

Let $u$ be the $L^2(D;\R^d)$-limit of the sequence $(u^n)_{n\ge 1}$. 
Applying first H\"older's inequality with respect to the measure $G(\cdot, y)dy$ and then Fubini's Theorem, we obtain:
\begin{align*}
&\left\|\int_D G_D^k(\cdot, y) (f(u^n(y))-f(u(y)))dy\right\|^2\\
& \qquad=\int_D \left(\int_D G_D^k(x,y)(f(u^n(y))-f(u(y)))dy\right)^2dx 
\\&\qquad\leq C \int_D\int_D G_D^k(x,y)(f(u^n(y))-f(u(y)))^2dydx 
\\&\qquad = C \int_D(f(u^n(y))-f(u(y)))^2\left(\int_D G_D^k(x,y)dx\right) dy
\\& \qquad\leq C\int_D(f(u^n(y))-f(u(y)))^2 dy.
\end{align*}
Since the operator $\operatorname{F}$ is strongly continuous, this yields
\beqn
\left\|\int_D G_D^k(\cdot, y) (f(u^n(y))-f(u(y)))dy\right\|^2\underset{n\to\infty}{\longrightarrow} 0.
\eeqn

Let $b\in \mathcal{S}$ be given by
\beqn
b(x)=\int_D G_D^k(x,y) g^i(y)dy+\sigma \int_DG_D^k(x,y) W(dy).
\eeqn
Consider a sequence $(b_n)_{n\geq 1}\in \mathcal{C}_0^{\infty}(D;\R^d)$, such that $b_n\rightarrow b$
in $L^2(D;\R^d)$. By taking limits in the $L^2(D;\R^d)$-norm in \eqref{unexpr}, we obtain 
that a.s., $u$ satisfies the system of equations defined in \eqref{mildPoisseq} on $L^2(D; \R^d)$. By continuity, for almost all $\omega\in \Omega$, this is also an identity for any $x\in D$.
\hfill{$\square$}

\end{proof}


\section{Properties of the solution}
\label{s2}

In this section we study the existence of moments of the solution to \eqref{mildPoisseq} and also the H\"older continuity of its sample paths. For this, we need a slightly stronger assumption than {\textbf{M}}, as follows.

\noindent $(\bar{\textbf{M}})$  The operator $F : L^2(D; \R^d)\rightarrow L^2(D; \R^d)$ defined in Section \ref{s1} admits a decomposition $F= F_1 + F_2$ which satisfies:
 \begin{enumerate}[(a)]
\item  $F_1 : L^2(D; \R^d)\rightarrow L^2(D; \R^d)$ is bounded. Moreover, for all $v,\, w\in L^2(D;\R^d)$,
\beqn
\langle v-w, F_1(v)-F_1(w)\rangle\geq 0;
\eeqn
\item  $F_2 : L^2(D; \R^d) \rightarrow L^2(D; \R^d)$ is such that there exists $L>0$ and for any $v\in L^2(D; \R^d)$,
\beqn
\vert F_2(v)(z_1)-F_2(v)(z_2)\vert \le L|v(z_1)-v(z_2)|, \ \forall z_1, z_2 \in \R^d.
\eeqn
\end{enumerate}
If $d=1$, the assumptions {\textbf{M}} and $(\bar{\textbf{M}})$ are equivalent. In general, $(\bar{\textbf{M}})$ implies {\textbf{M}} (with the same constant $L$).

\begin{prop}\label{up} 
The hypotheses are 
\begin{enumerate}[(i)]
\item $g\in L^2(D;\R^d)$;
\item The operator $\operatorname{F}$ satisfies the conditions  {\textbf{C}} and $(\bar{\textbf{M}})$;
\item Let $K=\sup_{x\in D}\|G_D^k(x,\cdot)\|_{L^2(D)}$. The constants $L$ and $a$ in $(\bar{\textbf{M}})$ and (\textbf{P}) respectively, satisfy $0<L<a\wedge (K|D|^{\frac{1}{2}})^{-1}$, where
$|D|$ denotes the Lebesgue measure of $D$.
\end{enumerate}
Then for all $p\geq 2$, 
\beqn
\|u\|_{L^p(\Omega;L^2(D;\R^d))} < \infty.
\eeqn
\end{prop}

\begin{proof}
From \eqref{mildPoisseq} and the triangular inequality, we have
\beqn
\|u\|_{L^p(\Omega;L^2(D))}\leq S_1+S_2+S_3,
\eeqn 
with 
\begin{align*}
S_1&=\left\|\int_D G_D^k(\cdot, y)f(u(y)) dy\right\|_{L^p(\Omega;L^2(D;\R^d))},
\\S_2&=\left\|\int_D G_D^k(\cdot, y)g(y) dy\right\|_{L^2(D;\R^d)},
\\S_3&=\left\|\int_D G_D^k(\cdot, y)W(dy)\right\|_{L^p(\Omega;L^2(D;\R^d))}.
\end{align*}

Let $f_j: \R^d \longrightarrow \R^d$ be defined as $f_j(v(y))= F_j(v)(y)$, for any $v\in L^2(D; \R^d)$, $y\in \R^d$, $j=1,2$. Then 
\begin{align*}
S_1&\leq\left \|\int_D G_D^k(\cdot, y)f_1(u(y))dy\right\|_{L^p(\Omega;L^2(D;\R^d))}\\
&+\left \|\int_D G_D^k(\cdot, y)f_2(u(y))dy\right\|_{L^p(\Omega;L^2(D;\R^d))}.
\end{align*}
For the first term on the right-hand side of this inequality we have,
\begin{align*}
&\left \|\int_D G_D^k(\cdot, y)f_1(u(y))dy\right\|_{L^p(\Omega;L^2(D;\R^d))}\\
&\qquad=\left(\E\left(\int_D dx \left|\int_D G_D^k(x,y)f_1(u(y))dy\right|^2\right)^{p/2}\right)^{1/p}
\\&\qquad\leq \left(\E\left(\int_D dx \left(\int_D G_D^k(x,y)^2dy\right)\int_D |f_1(u(y))|^2dy\right)^{p/2}\right)^{1/p}
\\&\qquad\leq M|D|^{1/2} \sup_{x\in D}\|G_D^k(x,\cdot)\|.
\end{align*}
For the second one, we use property (b) of $(\bar{\textbf{M}})$ to obtain
\begin{align*}
&\left \|\int_D G_D^k(\cdot, y)f_2(u(y))dy\right\|_{L^p(\Omega;L^2(D;\R^d))}\\
&\qquad\leq \left \|\int_D G_D^k(\cdot, y)f_2(0)dy\right\|_{L^p(\Omega;L^2(D;\R^d))}\\
&\qquad+\left \|\int_D G_D^k(\cdot, y)(f_2(u(y))-f_2(0))dy\right\|_{L^p(\Omega;L^2(D;\R^d))}\\
&\qquad\leq |f_2(0)|\sup_{x\in D}\|G_D^k(x,\cdot)\||D|\\
&\qquad+L\sup_{x\in D}\|G_D^k(x,\cdot)\||D|^{1/2}\|u\|_{L^p(\Omega;L^2(D;\R^d))}.
\end{align*}


By applying Schwarz's inequality, we have
\begin{align}
\label {S_2}
&S_2=\left(\int_D dx \left|\int_D G_D^k(x,y) g(y)dy\right|^2\right)^{1/2}\notag\\
&\leq \sup_{x\in D}\|G_D^k(x,\cdot)\|_{L^2(D)}|D|^{1/2}\Vert g\Vert.
\end{align}
Finally, we study $S_3$. We apply first H\"older's inequality with respect to the Lebesgue measure $dx$ , then the hypercontractivity property of Gaussian randon vectors and finally, the isometry property of the stochastic integral. This yields
\begin{align*}
\label{S_3}
S_3&=\left(\E\left(\int_D dx \left|\int_D G_D^k(x,y) W(dy)\right|^2\right)^{p/2}\right)^{1/p}\notag\\
&\leq |D|^{1/2-1/p}\left(\E\left(\int_D dx\left|\int_D G_D^k(x,y)W(dy)\right|^p\right)\right)^{1/p}\notag\\
&\leq C_p|D|^{1/2-1/p}\left(\int_D dx\left(\E\left|\int_D G_D^k(x,y)W(dy)\right|^2\right)^{p/2}\right)^{1/p}\notag\\
&\leq C_p|D|^{1/2}\sup_{x\in D}\|G_D^k(x,\cdot)\|.
\end{align*}
By Lemma \ref{l-sa-1}, $K=\sup_{x\in D}\|G_D^k(x,\cdot)\|_{L^2(D)}$ is finite. Hence, from the upper bounds proved so far we infer that 
\beqn
\|u\|_{L^p(\Omega;L^2(D))}\leq C_1 + L K |D|^{\frac{1}{2}} \|u\|_{L^p(\Omega;L^2(D))},
\eeqn
with $C_1= K|D|^{\frac{1}{2}}[M+|f_2(0)||D|^{\frac{1}{2}}+\Vert g\Vert+ C_p]$. Since we are assuming $1-L K |D|^{\frac{1}{2}}>0$, this yields the Proposition.
\hfill{$\square$}
\end{proof}

\begin{obs}
\label{smallLipschiz} 
In the context of elliptic operators, the assumption on the constant $L$ in the preceding Proposition is natural. It is a restriction to preserve the positiveness property of the operator $-\Delta +F$.
\end{obs}


\begin{prop}\label{thm2bl3}
The hypotheses are as in Proposition \ref{up}. Fix a ball centered at $0$ and with radius $r$, $B_r(0)$, strictly included in $D$. Then, 
for any $p\geq 2$ there exists a positive constant $C$ (depending on $r$)
such that, for any $x_1, x_2\in B_r(0)$, 
\beq
\label{p-s1-10}
\E[|u(x_1)-u(x_2)|^p]\leq C|x_1-x_2|^{p\xi},
\eeq
with 
\beqn
\xi=\begin{cases}
1, & \text {if}\  k=1,\\
1-\gamma, & \text{if}\ k=2,\\
\frac{1}{2}, & \text {if}\  k=3,
\end{cases}
\eeqn
 where $\gamma>0$ is arbitrarily small. Therefore, for almost all $\omega\in\Omega$, the sample paths of the process $u$ are H\"older continuous of degree 
 $\alpha\in(0,1)$, if $k=1,2$, and $\alpha\in(0,1/2)$, if $k=3$.
\end{prop}
\begin{proof}
From \eqref{mildPoisseq}, we clearly have
\beqn
\E[|u(x_1)-u(x_2)|^p]\leq C_p(A(x_1,x_2)+B(x_1,x_2)+C(x_1,x_2)),
\eeqn
where
\begin{align*}
A(x_1,x_2)&=\E\left[\left|\int_D[G_D^k(x_1,y)-G_D^k(x_2,y)]f(u(y))dy\right|^p\right],
\\B(x_1,x_2)&=\E\left[\left|\int_D[G_D^k(x_1,y)-G_D^k(x_2,y)]g(y)dy\right|^p\right],
\\C(x_1,x_2)&=\E\left[\left|\int_D[G_D^k(x_1,y)-G_D^k(x_2,y)]W(dy)\right|^p\right].
\end{align*}
The hypothesis $(\bar{\textbf{M}})$ implies the following:
\begin{align*}
\int_D|f(u(y))|^2dy&\leq 2 M^2+2\int_D|f_2(u(y))|^2 dy
\\&\leq 2M^2+4|D||f_2(0)|^2+4L\|u\|^2.
\end{align*}
Therefore, 
\begin{align}
\label{p-s1-101}
 \E\left[\int_D|f(u(y))|^2dy\right]^{p/2}&\leq C_1(M,|D|,f_2(0), p)+C_2(L,p)\E[\|u\|^p],
\end{align}
 and the right-hand side of this expression is finite, due to Proposition \ref{up}.
 
 Using this result and after applying Cauchy-Schwarz inequality, we obtain
 \begin{align}
\label{p-s1-11}
A(x_1,x_2)&\leq \left[\int_D[G_D^k(x_1,y)-G_D^k(x_2,y)]^2dy\right]^{p/2} \E\left[\int_D|f(u(y))|^2dy\right]^{p/2}\notag\\
&\le C\left[\int_D[G_D^k(x_1,y)-G_D^k(x_2,y)]^2dy\right]^{p/2}.
\end{align}
Applying again Cauchy-Schwarz inequality and using the properties of $g$ gives
\beq
\label{p-s1-12}
B(x_1,x_2)\leq C \left[\int_D[G_D^k(x_1,y)-G_D^k(x_2,y)]^2dy\right]^{p/2}.
\eeq
Finally, the hypercontractivity property yields
\begin{align}
\label{p-s1-13}
C(x_1,x_2)&\leq C \left(E\left|\int_D [G_D^k(x_1,y)-G_D^k(x_2,y)]W(dy)\right|^2\right)^{p/2}\notag\\
&= C \left(\int_D |G_D^k(x_1,y)-G_D^k(x_2,y)|^2dy\right)^{p/2}.
\end{align}
From \eqref{p-s1-11}--\eqref{p-s1-13} we see that
\beqn
\E[|u(x_1)-u(x_2)|^p]\leq C \left(\int_D |G_D^k(x_1,y)-G_D^k(x_2,y)|^2dy\right)^{p/2}.
\eeqn
We conclude the proof of \eqref{p-s1-10} by applying lemmas \ref{l-sa-2}, \ref{l-sa-3}, \ref{l-sa-4} of Section \ref{sa}. The statement on the sample paths of $u$ follows from Kolmogorov's continuity lemma.

\hfill{$\square$}
\end{proof}

\begin{prop}
\label{moments}
The hypotheses are as in Proposition \ref{up}. Then, for any $p\ge 2$, 
\beq
\label{p-s1-14}
\sup_{x\in D} \Vert u(x)\Vert_{L^p(\Omega; \mathbb{R}^d)} < \infty.
\eeq
\end{prop}
\begin{proof}
It is similar to that of the preceding proposition. By the triangular inequality,
\beqn
\E[|u(x)|^p]\leq C_p(A(x)+B(x)+C(x)),
\eeqn
with
\begin{align*}
A(x)&=\E\left[\left|\int_D G_D^k(x,y) f(u(y))dy\right|^p\right],
\\B(x)&=\E\left[\left|\int_D G_D^k(x,y) g(y)dy\right|^p\right],
\\C(x)&=\E\left[\left|\int_D G_D^k(x,y) W(dy)\right|^p\right].
\end{align*}
The conclusion will be obtained by proving that each of the above expressions are finite, uniformly in $x\in D$.
This relies on Lemma \ref{l-sa-1} and the following arguments.

Applying \eqref{p-s1-101} and Proposition \ref{up} yields
\beqn
\E\left[\int_D|f(u(y))|^2dy\right]^{p/2}\leq C.
\eeqn
Hence, by Cauchy-Schwarz inequality and Lemma \ref{l-sa-1},
\beqn
A(x)\leq \left[\int_D [G_D^k(x,y)]^2dy\right]^{p/2} \E\left[\int_D|f(u(y))|^2dy\right]^{p/2} \le C.
\eeqn
Similarly,
\beqn
B(x)\leq C \left[\int_D[G_D^k(x,y)]^2dy\right]^{p/2}\le C.
\eeqn
Finally, by the hypercontractivity property,
\begin{align*}
C(x)&\leq C \left(E\left|\int_D  G_D^k(x,y) W(dy)\right|^2\right)^{p/2}\notag\\
&= C \left(\int_D |G_D^k(x,y)|^2dy\right)^{p/2}\\
&\le C.
\end{align*}
In all the expressions above, $C$ denotes a finite constant. Hence \eqref{p-s1-14} holds.
\hfill{$\square$}

\end{proof}


\section{The law of the solution}
\label{s2-0}

This section is devoted to prove that the probability law of the solution to the system of SPDEs \eqref{mildPoisseq} is absolutely continuous with respect to a Gaussian measure defined on the Banach space $\mathcal{S}$. As a consequence, for any fixed $x\in D$, the law of $u(x)$ is absolutely continuous with respect to the Lebesgue measure on $\R^d$. To a large extent, the content of the section is an extension to the $d$--dimensional case of results proved in \cite{DM1992}.

Denote by $\mu$ the law on $\mathcal{S}$ of the Gaussian stochastic process
\beqn
\left(w(x) = \sigma\int_D G_D^k(x,y)W(dy), \ x\in D\right),
\eeqn
and by $H$ the Hilbert space $L^2(D; \R^d)$. Then, extending \cite[Proposition 3.1]{DM1992} for $d=1$, we have that $(\mathcal{S}, H, \mu)$ is an {\it abstract Wiener space}. Indeed,
$\mathcal{S}$ endowed with the supremum norm is a separable Banach space. By applying Schwarz inequality and then Lemma \ref{l-sa-1}, we obtain that the mapping 
\begin{align*}
i: H& \longrightarrow \mathcal{S}\\
h & \mapsto \int_D G_D^k(\cdot,y) h(y) dy
\end{align*} 
is continuous. Moreover, since the Dirichlet problem $\Delta v = h$ on $D$, $v|\partial D = 0$, has a unique solution, we have that the mapping $i$ is one-to-one and clearly, $i(H)$ is densely embedded in $\mathcal{S}$.
\smallskip

For its further use throughout this section, we introduce a new assumption.
\smallskip

\noindent{\bf(I)}  The function $f$ is continuously differentiable, and $\det J_f(x) \ne 0$, for any $x\in \R^d$, where $J_f$ denotes the Jacobian matrix of $f$. Moreover,
the linear operator $J_f^{-1}(\omega): H\rightarrow H$ defined by $J_f^{-1}(\omega)(h)(x)=J_f^{-1}(\omega(x))h(x)$ is {\it positive}, that is, 
\beqn
\left\langle J_f^{-1}(\omega)(h), h\right\rangle >0, \ \forall h\in H.
\eeqn
If $d=1$, the assumption is $f\in\mathcal{C}^1$, $f^\prime>0$ (see \cite[(H.1), p. 229]{DM1992}).
 
\begin{prop}
\label{p1-s2-0}
We keep the assumptions of Theorem \ref{thm31} and in addition, we suppose that {\bf(I)} holds. Then, the mapping
\begin{align*}
\bar F: \mathcal{S} & \longrightarrow i(H)\\
\omega & \mapsto i(F(\omega)) = \int_D G_D^k(\cdot,y) f(\omega(y)) dy,
\end{align*}
satisfies the following properties. 
\begin{enumerate}
\item For any $\omega\in\mathcal{S}$, there exists a Hilbert-Schmidt operator $D{\bar F}(\omega): H \longrightarrow H$ such that
\beq
\label{s2.1}
\Vert \bar F(\omega+i(h))-\bar F(\omega)- D\bar F(\omega)(h)\Vert_H = o(\Vert h\Vert_H),\ \text{as}\ \Vert h\Vert_H\to 0.
\eeq
\item For any $\omega\in\mathcal{S}$, the mapping $h\mapsto D\bar F(\omega+i(h))$ is continuous from $H$ into the space of Hilbert-Schmidt operators on $H$.
\item For any $\omega\in\mathcal{S}$, the mapping $I_H+ D\bar F(\omega)$ is invertible, where $I_H$ denotes the identity operator on $H$.
\end{enumerate}
\end{prop}
\begin{proof}
For any $\omega\in\mathcal{S}$, set
\beq
\label{s2-0-1}
D{\bar F}(\omega)(h) = J_f(\omega(\cdot))\int_D G_D^k(\cdot,y) h(y)dy, \ h\in H.
\eeq
The assumptions on $f$ imply that $\int_D |J_f(\omega(x))|^2 dx<\infty$. Then, by the definition of the Hilbert-Schmidt norm  (see e.g. \cite[Theorem VI.23, p. 210]{rs1}) and by using Lemma \ref{l-sa-1} we obtain,
\begin{align*}
\Vert D{\bar F}(\omega)\Vert^2_{HS}& =\int_D \int_D |J_f(\omega(x)) G_D^k(x,y)|^2 dx dy \\
& \le \sup_{x\in D}\left(\int_D (G_D^k(x,y))^2 dy\right)\left(\int_D |J_f(\omega(x))|^2 dx\right) < \infty.
\end{align*}
This yields that $D{\bar F}(\omega)$ is a Hilbert-Schmidt operator. 

From the expression \eqref{s2-0-1}, one checks that \eqref{s2.1} is satisfied. Moreover, from assertion 1. 
and the continuity of the map $J_f(\cdot)$, it is easy to verify that statement 2. holds.

For the proof of the third statement, we notice that the operator $D{\bar F}(\omega)$ is compact. Hence, by the Fredholm alternative it suffices to check that $\lambda=-1$ is not an eigenvalue. This fact is a consequence of the assumption {\bf (I)}. Indeed, if $\lambda=-1$ were an eigenvalue, there would exists a non null $h\in H$  satisfying 
\beqn
h + J_f(\omega)i(h)= 0.
\eeqn
 Equivalently,
\beqn
J_f^{-1}(\omega) h + \int_D G^k_D(\cdot,y) h(y)dy = 0. 
\eeqn
Take the inner product in $H$ with $h$ on each term of this identity. By property \textbf{(P)}, we obtain
\beqn
\left\langle J_f^{-1}(\omega) h ,h\right\rangle + a \left\Vert \int_D G_D^k(\cdot,y) h(y)dy\right\Vert^2 = 0.
\eeqn
By assumption {\bf (I)}, this implies that $h=0$.
\hfill{$\square$}
\end{proof}
\smallskip

In terms of  $\bar F$, the operator $\operatorname{T}$ defined in \eqref{defTop} is $\operatorname{T} = I_B + \bar F$. Hence, Proposition \ref{p1-s2-0} tell us that $\operatorname{T}$ satisfies the assumptions of \cite[Theorem 6.4]{K1982}. This yields the following result

\begin{prop}
\label{p2-s2-00} The hypotheses are as in Proposition 
\ref{p1-s2-0}. Denote by $\nu$ the law of $u = \operatorname{T}^{-1}(w)$.  Then, the probability $\nu$ is absolutely continuous with respect to $\mu$ (the law of $w$). Moreover, the density is given by
\beq
\label{s2-0-10}
\frac{d\nu}{d\mu}(\omega) = \left\vert {\det}_2(I_H + D \bar F(\omega))\right\vert \exp\left(-\delta(f(\omega)) - \frac{1}{2}\Vert f(\omega)\Vert_H^2\right),
\eeq
where $\det_2$ denotes the Carleman-Fredholm determinant, and $\delta$ denotes the divergence operator, also called the Skorohod integral operator (see \cite[Theorem 5.8.3]{Bo98} for a definition of this notion in this context).
\end{prop}

\begin{obs}
\label{r-s2-1}
For any $x\in D$, let $\pi_x: \mathcal{S}\longrightarrow \R^d$ be defined by $\pi_x(\omega)= \omega(x)$. Clearly, $\nu \ll \mu$ implies $\nu\circ \pi_x^{-1}\ll \mu\circ \pi_x^{-1}$. Since $\mu\circ \pi_x^{-1}$ is the law of the random vector $w(x)$, which is Gaussian, we infer that the law of $u(x)$ is absolutely continuous with respect to the Lebesgue measure on $\R^d$. 
\end{obs}

\section{Gaussian solutions}
\label{s3}

In this section we consider the system \eqref{mildPoisseq} in the particular case $f=g=0$. Under this assumption, \eqref{mildPoisseq} gives an explicit expression of the solution, which clearly defines the $d$--dimensional Gaussian random vector:

\beq
\label{s3-1}
w^i(x) = \sum_{j=1}^d \sigma_{ij} \int_D G_D^k(x,y) W^j(dy), \  x\in D, \ i = 1, \ldots, d.
\eeq
We are assuming that $\sigma = (\sigma_{ij})_{1\le i,j\le d}$ is a non-singular matrix. Therefore, without loss of generality, we can reduce the analysis of the stochastic process given in \eqref{s3-1} to
the case where $\sigma$ is the identity matrix in $\R^d$. By doing so, we are left to consider the Gaussian vector $v(x) = (v^i(x))_i$ with independent, identically distributed components
defined by
\beq
\label{s3-11}
v^i(x) = \int_D G_D^k(x,y) W^i(dy), \  x\in D, \ i = 1, \ldots, d.
\eeq
Its density is given by the formula
\beq
\label{s3-2}
p_{v(x)} (z) = (2\pi \sigma_x^2)^{-\frac{d}{2}} \exp\left(-\frac{|z|^2}{2\sigma_x^2}\right), \ z\in\R^d,
\eeq
where $\sigma_x = \Vert G_D^k(x,\cdot)\Vert$. 

According to Corollary \ref{c-sa-1}, the mapping $x\in D\mapsto \sigma_x$ is continuous and therefore, $\inf_{x\in K} \sigma_x$ and $\sup_{x\in K} \sigma_x$ are both achieved on any compact subset $K\subset D$. Let $x_0, x_1\in K$ be such that 
\beq
\label{s3-3}
0<\sigma_{x_0} = \inf_{x\in K} \sigma_x \le \sup_{x\in K} \sigma_x = \sigma_{x_1} < \infty.
\eeq
Then, 
\beq
\label{s3-5}
\sup_{(z,x)\in\R^d\times K} p_{v(x)}(z) \le \left(2\pi \sigma_{x_0}^2\right)^{-\frac{d}{2}} < \infty,
\eeq
and for any compact set $\tilde K\subset \R^d$, 
\beq
\label{s3-500}
c_1\left(2\pi \sigma_{x_1}^2\right)^{-\frac{d}{2}} \le \inf_{(z,x)\in\tilde K\times K} p_{v(x)}(z),
\eeq 
where $c_1 = \inf_{z\in \tilde K}\exp\left(-\frac{|z|^2}{2\sigma_{x_0}^2}\right)$.
\medskip

\subsection{Sample paths of the process $v$}
\label{s3-1
}
From Theorem \ref{thm2bl3}, we already know that the sample paths of the Gaussian process defined by \eqref{s3-1} are H\"older continuous.
However, under the standing assumptions, more can be said. 
\smallskip

\noindent{\em Case $k=1$}
\smallskip

The trajectories of $\{v(x), x\in (0,b)\}$ are differentiable, a.s. Indeed, from the expression \eqref{GD1} and by applying the It\^o formula we have,
\beq
\label{tk1}
v^i(x) = \frac{x}{b} \int_0^b W^i(y) dy - \int_0^x W^i(y) dy, \ i=1,\ldots, d.
\eeq
(see \cite[Lemma 2.1]{BP1990}). 

\noindent{\em Case $k=2, 3$}
\smallskip

Let $D=B_1(0)$ and $D_0=B_{\rho_0}(0)$ with $\rho_0<1$. For any $x,y\in D$ and $\gamma$ arbitrarily small, define
\beqn
\tau(x,y) = 
\begin{cases}
|x-y|^{1-\gamma},  & \ {\text {if}}\ k=2,\\
|x-y|^{\frac{1}{2}}, & \ {\text {if}}\ k=3.
\end{cases}
\eeqn
According to the discussion in \cite[p. 164-167]{X2009}, and by applying the estimates \eqref{sa-32}, \eqref{sa-321} (for $k=2$) and \eqref{sa-42} (for $k=3$), we have the following results on the {\em uniform modulus of continuity} of the process $v$.
\begin{description}
\item {(1)} Extensions of the classical Garsia-Rodemich-Rumsey Lemma (see \cite[Theorems 4.1, 4.2]{X2009}) yield the existence of a random variable $A$ having moments of any order, such that, for any $x,y \in D_0$,
\beq
\label{grr}
|v(x)-v(y)| \le A \tau(x,y) \sqrt{\log(1+\tau(x,y)^{-1})}.
\eeq
\item{(2)} From results in \cite{KP2006}, one can obtain more information on the random variable $A$. Indeed, there exists a constant $c>0$ such that
\beq
\label{kr}
\E\left\{\exp\left(c\sup_{x,y\in D_0} \frac{|v(x)-v(y)|^2}{\log(1+\tau(x,y)^{-1})}\right)\right\} < \infty
\eeq
(see \cite[Corollary 4.4]{X2009}).
\item{(3)} By using entropy methods and the Gaussian isoperimetric inequality, we obtain
\beq
\label{mc1}
\limsup_{|h|\to 0}\frac{\sup_{x\in D_0, y\in B_h(0)}|v(x+y)-v(x)|}{\tau(0,y)\sqrt{\log(1+\tau(0,y)^{-1})}}\le C,
\eeq
where $C$ is a finite positive constant. Whether this estimate is sharp is an open question.
\end{description}
Clearly, the above results yield H\"older continuity of the sample paths, a.s.

Let $k=2$. Using Lemma \ref{l-sa-3} and arguing as in \cite[Chapter 5]{DSS2009}, we deduce the property:

{\sl For almost all $\omega$, the sample paths of the process $\{v(x), x\in D_0\}$ are H\"older continuous of degree $\alpha\in(0,1)$, though there are not Lipschitz continuous}.

Similarly, for $k=3$, using Lemma \ref{l-sa-4}, we have:

{\sl For almost all $\omega$, the sample paths of the process $\{v(x), x\in D_0\}$ are H\"older continuous of degree $\alpha\in(0,1/2)$, though there are not for $\alpha>1/2$}.

\subsection{Joint densities}
\label{S3-2}

For $k=1$, $D_0$ denotes a closed interval of $D= (0,b)$, and as in the previous section, for $k=2,3$, $D_0 = B_{\rho_0}(0)$, with $\rho_0\in(0,1)$. In this section we prove the following facts:
\begin{description}
\item{(a)} ${\text{Var}}\ v^i(x)>0$, for any $i=1,\ldots,d$, $x\in D_0$.
\item{(b)} ${\text{Corr}}(v^i(x_1), v^i(x_2))<1$, for any $i=1,\ldots,d$, and for each $x_1, x_2\in D_0$. 
\end{description}
We recall that, for any $i=1,\ldots,d$, 
\beqn
 {\text{Var}}\ v^i(x) = \Vert G_D^k(x,\cdot)\Vert^2:=\sigma_x^2,
 \eeqn
 and we will use the following notations:
 \begin{align*}
\sigma_{x_1,x_2}&:={\text{Cov}}(v^i(x_1), v^i(x_2)) = \left\langle G_D^k(x_1,\cdot), G_D^k(x_2,\cdot)\right\rangle,\\
\rho_{x_1,x_2}&:={\text{Corr}}(v^i(x_1), v^i(x_2)) = \frac{\sigma_{x_1,x_2}}{\sigma_{x_1}\sigma_{x_2}}.
\end{align*}
Because of the independence of the components of $v(x_1)$ and of $v(x_2)$, properties (a) and (b) imply the existence of joint density of the $2d$-dimensional vector 
\beqn
(v(x_1), v(x_2)), \ x_1, x_2\in D_0.
\eeqn

Property (a) follows trivially from \eqref{s3-3}. As for property (b), it is a consequence of property (a) and the next lemma.

\begin{lema}
\label{l-s4-1}
For any $x_1, x_2\in D_0$, $x_1\ne x_2$, we have
\beq
\label{s3-40}
\sigma_{x_1}^2\sigma_{x_2}^2 - \sigma^2_{x_1,x_2}>0.
\eeq
\end{lema}
\begin{proof}
We argue by contradiction. Assume that $\sigma_{x_1}^2\sigma_{x_2}^2 - \sigma^2_{x_1,x_2}=0$.
Then, $\lambda\in \R\setminus\{0\}$ (depending on $x_1$, $x_2$) would exist satisfying $v(x_1) = \lambda v(x_2)$.
This implies $\Vert G_D^k(x_1,\cdot)- \lambda G_D^k(x_2,\cdot)\Vert = 0$ or equivalently, $G_D^k(x_1,y)- \lambda G_D^k(x_2,y)=0$, for almost every $y$ (with respect to the Lebesgue measure). 
\smallskip

\noindent{\em Case $\lambda=1$}. The condition $\Vert G_D^k(x_1,\cdot)- G_D^k(x_2,\cdot)\Vert = 0$ 
yields a contradiction with the lower bounds given in \eqref{sa-21},
\eqref{sa-321}, \eqref{sa-42}, for $k=1$, $k=2$, $k=3$, respectively.
\smallskip

\noindent {\em Case $\lambda\neq1$}. The condition $G_D^k(x_1,y)- \lambda G_D^k(x_2,y)=0$, for almost every $y$ implies that 
for any $f\in L^2(D)$,
\beq
\label{s3-41}
\langle G_D^k(x_1,y)- \lambda G_D^k(x_2,y), f\rangle=0.
\eeq
Moreover,
\beq
\label{s3-42}
\Vert G_D^k(x_1,\cdot)\Vert = \lambda \Vert G_D^k(x_2,\cdot)\Vert.
\eeq
With this identity, and by developing the square of $\Vert G_D^k(x_1,\cdot)- \lambda G_D^k(x_2,\cdot)\Vert$, we obtain
 \beq
\label{s3-43}
\lambda \Vert G_D^k(x_2, \cdot)\Vert^2 = \langle G_D^k(x_1,\cdot), G_D^k(x_2,\cdot)\rangle.
\eeq
Choose $f = \nu G_D^k(x_1,\cdot)-G_D^k(x_2, \cdot)$, with $\nu\in \R$ to be determined later. The identity \eqref{s3-41} implies,
\begin{align}
\label{s3-44}
0 & = \langle G_D^k(x_1,\cdot)- \lambda G_D^k(x_2,\cdot), \nu G_D^k(x_1,\cdot)-G_D^k(x_2, \cdot)\rangle\notag\\
& = \Vert G_D^k(x_1,\cdot)- G_D^k(x_2,\cdot)\Vert^2 + (1-\lambda) \langle G_D^k(x_1,\cdot), G_D^k(x_2,\cdot)\rangle\notag\\
& \quad -(1-\lambda) \Vert G_D^k(x_2, \cdot)\Vert^2 + (\nu-1)\Vert G_D^k(x_1, \cdot)\Vert^2\notag\\
& \quad -(\nu-1)\langle G_D^k(x_1,\cdot), G_D^k(x_2,\cdot)\rangle.
\end{align}
By  applying \eqref{s3-42}, \eqref{s3-43} to \eqref{s3-44}
we obtain
\begin{align}
\label{s3-45}
0 & = \langle G_D^k(x_1,\cdot)- \lambda G_D^k(x_2,\cdot), \nu G_D^k(x_1,\cdot)-G_D^k(x_2, \cdot)\rangle\notag\\
&  =\Vert G_D^k(x_1,\cdot)- G_D^k(x_2,\cdot)\Vert^2 \notag\\
& \quad + (\lambda-1)[\lambda(\nu-2)+1] \Vert G_D^k(x_2,\cdot)\Vert^2.
\end{align}
Assume first that $\lambda>1$. By choosing $\nu>2-\frac{1}{\lambda}$, the factor $(\lambda-1)[\lambda(\nu-2)+1]$ in \eqref{s3-45} is positive. Hence, we obtain
\begin{align*}
0 & = \langle G_D^k(x_1,\cdot)- \lambda G_D^k(x_2,\cdot), \nu G_D^k(x_1,\cdot)-G_D^k(x_2, \cdot)\rangle\notag\\
& \ge \Vert G_D^k(x_1,\cdot)- G_D^k(x_2,\cdot)\Vert^2,
\end{align*}
which, arguing as for the case $\lambda=1$, yields a contradiction.

If $\lambda < 1$, we choose $\nu<2-\frac{1}{\lambda}$ to obtain that $(\lambda-1)[\lambda(\nu-2)+1]>0$. Similarly as above,  we arrive at a contradiction.

The proof of \eqref{s3-40} is complete.
 \hfill{$\square$}
\end{proof}

\begin{lema}
 \label{l-s4-2}
Let $m_{x_1,x_2}=\frac{\sigma_{x_1,x_2}}{\sigma_{x_1}^2}$ be the conditional mean of $v^i(x_2)$ given $v^i(x_1)$, $i=1,\ldots, d$.
 Then, there exists a constant $C>0$ such that for all $x_1, x_2\in D_0$,
 \beq
 \label{s3-46}
 \left\vert 1-m_{x_1,x_2}\right\vert \le C \Vert v(x_1)-v(x_2)\Vert_{L^2(\Omega; \R^d)}.
 \eeq
  \end{lema}
  The proof of this lemma follows easily from the definition of $m_{x_1,x_2}$. We refer the reader to \cite[p. 1359]{DSS2010} for details.


 \begin{lema}
 \label{l-s4-3}
Let $\tau^2_{x_1,x_2}= \sigma_{x_2}^2(1-\rho_{x_1,x_2}^2)$ be the conditional variance of $v^i(x_2)$ given $v^i(x_1)$, $i=1,\ldots,d$. 
 Then, there exists a constant $C_2>0$ such that for all $x_1, x_2\in D_0$,
 \beq
 \label{s3-47}
 \tau_{x_1,x_2} \le C_2 \Vert v(x_1)-v(x_2)\Vert_{L^2(\Omega; \R^d)}.
 \eeq
 \end{lema}
 \begin{proof}
For any $x_1, x_2\in D$, let
\beq
\label{metric}
\delta(x_1,x_2) := \Vert v(x_1)-v(x_2)\Vert_{L^2(\Omega; \R^d)} = \Vert G(x_1,\cdot) - G(x_2,\cdot)\Vert,
\eeq
be the canonical pseudo-metric associated with the Gaussian process $v$. 

With simple computations, we obtain
\begin{align}\label{sigmas}
\sigma^2_{x_2}\sigma^2_{x_1}-\sigma_{x_1,x_2}^2
& = \frac{1}{4}\left[\delta(x_1,x_2)^2-(\sigma_{x_2}-\sigma_{x_1})^2\right]\left[(\sigma_{x_2}+\sigma_{x_1})^2-\delta(x_1,x_2)^2\right].
\end{align}
By the triangular inequality,
\begin{align*}
(\sigma_{x_1}-\sigma_{x_2})^2 &= \left\vert  \Vert G(x_1,\cdot)\Vert - \Vert G(x_2,\cdot)\Vert\right\vert^2\\
& \le \Vert G(x_1,\cdot)- G(x_2,\cdot)\Vert^2
= \delta(x_1,x_2)^2.
\end{align*}
Hence, the first factor on the right-hand side of \eqref{sigmas} is nonnegative. Moreover, 
we have proved in Lemma \ref{l-s4-1} that $1-\rho_{x_1,x_2}^2>0$. Hence, using \eqref{s3-3}, we have the following upper bounds:
\begin{align}
\label{s3-48}
1-\rho_{x_1,x_2}^2& \le C\left[\delta(x_1,x_2)^2-(\sigma_{x_2}-\sigma_{x_1})^2\right](\sigma_{x_2}+\sigma_{x_1})^2\notag\\
& \le C\left\{\delta(x_1,x_2)^2\left(\sigma^2_{x_2}+\sigma^2_{x_1}\right) +(\sigma^2_{x_2}-\sigma^2_{x_1})^2\right\}\notag\\
& \le C \left[\delta(x_1,x_2)^2 + (\sigma_{x_1}-\sigma_{x_2})^2\right]\notag\\
& \le C\delta(x_1,x_2)^2
\end{align}
The inequality \eqref{s3-47} is a consequence of \eqref{s3-3} and \eqref{s3-48}.
\hfill{$\square$}

\end{proof}

\subsection{Upper and lower bounds of the canonical metric}
\label{sa}

In this section, we prove upper and lower bounds for the canonical pseudo-metric relative to the Gaussian process $v$ given in \eqref{metric}. This is equivalent to establish bounds from above and from below for $\Vert G_D^k(x_1,\cdot) - G_D^k(x_2,\cdot)\Vert$. 

\begin{lema}\label{l-sa-2}
Let $k=1$ and $D=(0,b)$, $b>0$. For any $x_1,\, x_2\in D$, we have
\beq
\label{sa-21}
\left(\frac{b}{3}\right)^{\frac{1}{2}}|x_1-x_2|\le \|G_D^1(x_1,\cdot)-G_D^1(x_2,\cdot)\|\leq \left(\frac{7b}{3}\right)^{\frac{1}{2}}|x_1-x_2|.
\eeq
\end{lema}
\begin{proof}
Using the expression \eqref{GD1}, we clearly have
\beqn
\|G_D^1(x_1,\cdot)-G_D^1(x_2,\cdot)\|^2 = |x_1-x_2|^2 \int_0^b\left(\frac{x_2\wedge y-x_1\wedge y}{x_2-x_1}-\frac{y}{b}\right)^2dy.
\eeqn
The integral on the right-hand side of this equality is $\frac{b}{3}+\frac{x_2^2+x_2 x_1+x_1^2}{3b}+x_1$. On $(0,b)$ this expression is bounded from above by $\frac{7b}{3}$, and from below by $\frac{b}{3}$. This yields \eqref{sa-21}.
\hfill{$\square$}
\end{proof}
\medskip

For $k=2, 3$, $D= B_1(0)$, we will use the following formulas for the Green function (see for instance \cite[[pg. 19]{GT2001}):
\begin{align}
G_D^k (x,y)& = \Gamma^k((|x-y|) - \Gamma^k\left[|y|\left\vert x-\frac{y}{|y|^2}\right\vert\right], & y\ne 0,\label {sa-31}\\
G_D^k (x,y)& = \Gamma^k(|x|) - \Gamma^k(1), &  y=0. \notag
\end{align}
with $\Gamma^k$ defined in \eqref{sa-310}.

For every $x,y\in D$, define
\beq\label{sa-311}
L^k_x(y) = \Gamma^k(|x-y|), \ \  S^k_x(y)= \Gamma^k\left[|y|\left\vert x-\frac{y}{|y|^2}\right\vert\right],
\eeq
so that for $y\ne 0$,
\beq
\label{sa-3111}
G_D^k (x,y) = L^k_x(y) - S^k_x(y).
\eeq

Notice that for any $y\in D$, $x\rightarrow S^k_x(y)$ is a harmonic function, and $S^k_x(y)=L^k_x(y)$ for $y\in \p D$. 

Clearly, for any $x_1, x_2\in D$,
\beqn
\|G_D^k(x_1,\cdot)-G_D^k(x_2,\cdot)\|\leq \|L^k_{x_1}-L^k_{x_2}\|+\|S^k_{x_1}-S^k_{x_2}\|.
\eeqn
\smallskip


\begin{lema}\label{l-sa-3}
Let $k=2$ and $D=B_1(0)$. Fix $\rho_0<1$. 
\begin{enumerate}
\item There exists a positive constant $C$ such that
\beq
\label{sa-32}
\|G_D^2(x_1,\cdot)-G_D^2(x_2,\cdot)\|\leq C|x_1-x_2|\left\vert\log^2(|x_1-x_2|) - \log (|x_1-x_2|) + 1\right\vert^{\frac{1}{2}},
\eeq
for any $x_1,\, x_2\in\bar{B}_{\rho_0}(0)$. The constant $C$ above is of the form 
$\frac{c}{(1-\rho_0)^2}$ where $c>0$ is a multiple of $\pi^{-\frac{1}{2}}$.

Therefore,
\begin{align}
\label{sa-320}
&\|G_D^2(x_1,\cdot)-G_D^2(x_2,\cdot)\|\notag\\
&\qquad \leq C\left[|x_1-x_2||\log(|x_1-x_2|)|1_{\{|x_1-x_2|\le e^{-1}\}} + |x_1-x_2|1_{\{|x_1-x_2|> e^{-1}\}}\right], 
\end{align}
for any $x_1,\, x_2\in\bar{B}_{\rho_0}(0)$, where $C$ is a constant of the same type as in \eqref{sa-32}.

\item There exists a positive constant $\bar C$ such that
\beq
\label{sa-321}
\|G_D^2(x_1,\cdot)-G_D^2(x_2,\cdot)\|\ge \bar C|x_1-x_2|,
\eeq
for any $x_1,\, x_2\in\bar{B}_{\rho_0}(0)$. The constant $\bar C$ above is a multiple of $\pi^{-\frac{1}{2}}$.
\end{enumerate}

\end{lema}
\begin{proof}
First, we will prove an upper bound for $\|L^2_{x_1}-L^2_{x_2}\|_2$. 
Let $r_{x_1,x_2}= 2|x_1-x_2|$. Assume  $|x_1-x_2|>1$. Then $|y-x_1|\le r_{x_1,x_2}$, for any $|y|\le 1$, and
\beqn
\|L^2_{x_1}-L^2_{x_2}\|^2 \le (2\pi^2)^{-1}[J_1(x_1) + J_2(x_2)],
\eeqn
with 
\begin{align*}
J_1(x_1) &= \int_{\{|y|\leq 1\}\cap\{|y-x_1|\leq r_{x_1,x_2}\}}\log^2 |x_1-y| dy,\\
J_2(x_2) &= \int_{\{|y|\leq 1\}\cap\{|y-x_2|\leq \frac{3r_{x_1,x_2}}{2}\}} \log^2 |x_2-y|dy.
\end{align*}
Using polar coordinates $(r,\theta)$ and a change of variables $\rho=r^2$, we have
\begin{align*}
J_1(x_1)  \le \frac{\pi}{4} \int_0^{r_{x_1,x_2}^2} (\log^2 \rho)\  d\rho
=\pi r_{x_1,x_2}^2\left(\log^2\left[\frac{1}{r_{x_1,x_2}}\right]+\log\left[\frac{1}{r_{x_1,x_2}}\right]+\frac{1}{2}\right),
\end{align*}
where the integral is computed using integration by parts.

Similarly,
\beqn
J_2(x_1) \le c r_{x_1,x_2}^2\left(\log^2\left[\frac{1}{r_{x_1,x_2}}\right]+\log\left[\frac{1}{r_{x_1,x_2}}\right]+\frac{1}{2}\right),
\eeqn
with a constant c which is a multiple of $\pi$ and, consequently
\beq
\label{sa-33}
\|L^2_{x_1}-L^2_{x_2}\|^2 \le C |x_1-x_2|^2\left(\log^2 \left[\frac{1}{|x_1-x_2|}\right] + \log\left[\frac{1}{|x_1-x_2|}\right]+1\right),
\eeq
for some positive constant $C$ which is a multiple of $\pi^{-1}$.

Next, we assume that $|x_1-x_2|\le 1$. We have
\beq
\label{sa-34}
\|L^2_{x_1}-L^2_{x_2}\|^2 \leq \pi^{-2}[J_1(x_1)+J_2(x_2))] + (2\pi^2)^{-1}J_3(x_1,x_2),
\eeq
with
\beqn
 J_3(x_1,x_2) = \int_{\{|y|\leq 1\}\cap \{|y-x_1|> r_{x_1,x_2}\}\}}\left(\log |x_1-y|-\log|x_2-y|\right)^2 dy.
 \eeqn
 Let $\varphi(\lambda)=\log(|\lambda(x_2-y)+(1-\lambda)(x_1-y)|)$, $\lambda\in (0,1)$. Then, 
\beqn
\log|x_2-y| - \log|x_1-y|=\varphi(1)-\varphi(0)=\int_0^1 \varphi'(\lambda)d\lambda.
\eeqn
Denote by $\alpha_\lambda$ the angle between the vectors $x_1-x_2$ and $\lambda(x_2-y)+(1-\lambda)(x_1-y)$. Direct computations show that 
\beqn
\varphi'(\lambda)=\frac{|x_1-x_2|\cos(\alpha_\lambda)}{|\lambda(x_2-y)+(1-\lambda)(x_1-y)|}, 
\eeqn
Hence,
\begin{align*}
& J_3(x_1,x_2)\le |x_1-x_2|^2\\
&\quad \times \int_{\{|y|\leq 1\}\cap \{|y-x_1|> r_{x_1,x_2}\}\}} dy\left(\int_0^1\frac{1}{|\lambda(x_2-y)+(1-\lambda)(x_1-y)|} d\lambda\right)^2.
\end{align*}
On $\{|y-x_1|> r_{x_1,x_2}\}$,
\beqn
|\lambda(x_2-y)+(1-\lambda)(x_1-y)|\geq |y-x_1|-\lambda|x_2-x_1|\ge \frac{|y-x_1|}{2}.
\eeqn
Therefore,
\begin{align}
\label{sa-35}
 J_3(x_1,x_2) &\le 4 |x_1-x_2|^2  \int_{\{|y|\leq 1\}\cap \{|y-x_1|> r_{x_1,x_2}\}} |y-x_1|^{-2} dy \notag\\
& \le 8\pi |x_1-x_2|^2 \log\left[\frac{1}{|x_1-x_2|}\right].
 \end{align}
 From \eqref{sa-33}--\eqref{sa-35}, we have
\beq
\label{sa-36}
\|L^2_{x_1}-L^2_{x_2}\|^2 \leq  C |x_1-x_2|^2\left(\log^2 \frac{1}{|x_1-x_2|} + \log\frac{1}{|x_1-x_2|}+1\right),
\eeq
with a positive constant which is a multiple of $\pi^{-1}$.

For the study of the contribution of $\|S^2_{x_1}-S^2_{x_2}\|_2$ it is useful to identify $\R^2$ with $\mathbb{C}$ (the set of complex numbers) and to consider the following identity:
\beqn
|y|\left\vert x-\frac{y}{|y|^2}\right\vert = \left\vert 1-\bar xy\right\vert,
\eeqn
where $\bar x$ denotes the conjugate of the complex number $x$. By doing so, it is easy to check that
\beq
\label{sa-361}
2\pi\vert \nabla_x S^2_{\cdot}(y)\vert = \vert \nabla_x \log(|1-\bar x y|)\vert \le \frac{|y|}{\sqrt {2}(1-|x|)^2}.
\eeq
By the mean value theorem, this implies,
\begin{align}
\label{sa-362}
\|S^2_{x_1}-S^2_{x_2}\|^2& = (2\pi)^{-2}\int_{|y|\leq 1}|\log|1-\bar{x}_1 y|-\log|1-\bar{x}_2 y||^2 dy\notag
\\& \le (8\pi^2)^{-1} |x_1-x_2|^2\int_{|y|\leq 1}\left(\frac{1}{(1-|x^*|)^2}\right)^2dy, 
\end{align}
with $x^*=\lambda x_1+(1-\lambda)x_2$.
We are assuming $x_1, x_2 \in \bar B_{\rho_0}(0)$ with $\rho_0<1$. Hence, $1-|x^*|\ge 1-\rho_0$ and therefore,
\beq
\label{sa-37}
\|S^2_{x_1}-S^2_{x_2}\|^2 \le \frac{C}{(1-\rho_0)^4}|x_1-x_2|^2,
\eeq
with a constant C which is a multiple of $\pi^{-1}$.

With \eqref{sa-36}, \eqref{sa-37}, we have proved \eqref{sa-32}.

If $|x_1-x_2| \leq e^{-1}$, then 
$\left\vert\log^2(|x_1-x_2|) - \log (|x_1-x_2|) + 1\right\vert \le 3 \log^2(|x_1-x_2|)$,
while if $|x_1-x_2| > e^{-1}$, 
\beqn
\sup_{e^{-1}<|x_1-x_2|\le 2}\left[\left\vert\log^2(|x_1-x_2|) - \log (|x_1-x_2|) + 1\right\vert\right] \le C.
\eeqn
Therefore \eqref{sa-32} clearly implies \eqref{sa-320}.
\hfill{$\square$}


Next, we prove \eqref{sa-321}. Let 
$\eta\in\left(0,\frac{1-\rho_0}{2\rho_0}\right)$. Since $|x_1-x_2|\le 2\rho_0$,  we have $\eta |x_1-x_2|<1-\rho_0$. Let $D_{\eta} = \{y\in D: |y-x_1|<\eta |x_1-x_2|\}$. The choice of $\eta$ implies $D_\eta\subset D$, and then,
\begin{align*}
\|G_D^2(x_1,\cdot)-G_D^2(x_2,\cdot)\|^2 &\ge \|G_D^2(x_1,\cdot)-G_D^2(x_2,\cdot)\|_{\eta}^2\\
&\ge \frac{1}{2}\
\Vert L_{x_1}^2-L_{x_2}^2\Vert_\eta^2
-\Vert S_{x_1}^2-S_{x_2}^2\Vert_\eta^2,
\end{align*}
where $\Vert\cdot\Vert_\eta$ denotes the $L^2$-norm on $D_\eta$.
Similarly as in \eqref{sa-362}, using \eqref{sa-361}, se have
\begin{align}
\label{sa-371}
\Vert S_{x_1}^2-S_{x_2}^2\Vert_\eta^2 & = \int_{D_\eta}(S^2_{x_1}(y)-S^2_{x_2}(y))^2 dy\notag\\
&=|x_1-x_2|^2 \int_{D_\eta}|\nabla_{x^*} S^2_{\cdot}(y)|^2 dy\notag\\
&\le (8\pi^2)^{-1} |x_1-x_2|^2 (1-\rho_0)^{-4}\int_{D_\eta} |y|^2 dy\notag\\
&\le c\pi^{-1}(1-\rho_0)^{-4}\eta^2|x_1-x_2|^4.
\end{align}
We continue the proof by establishing a lower bound for $\Vert L_{x_1}^2-L_{x_2}^2\Vert_\eta^2$. For this, we take a new domain of integration $\bar D_\eta\subset D_\eta$ defined as the intersection of the set 
\beqn
C_\eta = \{y\in D: \frac{\eta}{2}|x_1-x_2|<|y-x_1|<\eta|x_1-x_2|\}
\eeqn
with the points $y\in D_\eta$ such that the angle between the lines joining $x_1$ with $y$ and $x_1$ with $x_2$ lies in the interval $(-\pi/4,\pi/4)$. Then, similarly as in the study of the term $J_3(x_1,x_2)$ above, we obtain
\begin{align*}
\Vert L_{x_1}^2-L_{x_2}^2\Vert_\eta^2 &\ge
 (2\pi)^{-2} \int_{\bar D_{\eta}} \left(\log |x_1-y|-\log|x_2-y|\right)^2 dy\\
 & = (2\pi)^{-2} |x_1-x_2|^2\\
 &\quad \times \int_{\bar D_{\eta}}\left\vert \int_0^1 \frac{cos (\alpha_\lambda)}{|\lambda(x_2-y)+ (1-\lambda)(x_1-y)|} d\lambda\right\vert^2 dy.
\end{align*}
Remember that $\alpha_\lambda$ stands for the angle between the vectors $x_1-x_2$ and $\lambda(x_2-y)+ (1-\lambda)(x_1-y) = x_1-y+\lambda(x_2-x_1)$.  Also observe that, on $\bar D_{\eta}$, we have
$1/\sqrt{2} \le cos (\alpha_\lambda)\le 1$, and  $|y-[x_1+\lambda(x_2-x_1)]|\le |y-x_2|$. Hence, from the above inequalities, we have
\beqn
\Vert L_{x_1}^2-L_{x_2}^2\Vert_\eta^2 
\ge (8\pi^2)^{-1}  |x_1-x_2|^2\int_{\bar D_{\eta}}\frac{dy}{|y-x_2|^2}.
\eeqn
After the change of variables defined by $y\mapsto \frac{1}{2}(y-x_2)$ and then by using polar coordinates, we have
\beqn
\int_{\bar D_{\eta}}\frac{dy}{|y-x_2|^2} = C \pi \int_{\frac{\eta}{4}|x_1-x_2|}^{\frac{\eta}{2}|x_1-x_2|}\frac{dr}{r} = C \log 2.
\eeqn
Thus,
\beq
\label{sa-372}
\Vert L_{x_1}^2-L_{x_2}^2\Vert_\eta^2 \ge C\pi^{-1}  |x_1-x_2|^2.
\eeq
Along with \eqref{sa-371} this yields
\beqn
\|G_D^2(x_1,\cdot)-G_D^2(x_2,\cdot)\|^2  \ge C\pi^{-1} |x_1-x_2|^2 \left(1- (1-\rho_0)^{-4}4\eta^2\right).
\eeqn
Finally, by choosing $\eta\in\left(0,\frac{(1-\rho_0)^2}{2\sqrt 2}\wedge \frac{1-\rho_0}{2\rho_0}\right)$, we see that 
\beqn
\|G_D^2(x_1,\cdot)-G_D^2(x_2,\cdot)\|^2  \ge C\pi^{-1} |x_1-x_2|^2,
\eeqn
proving \eqref{sa-321}.
\hfill{$\square$}
\end{proof}
\smallskip

\begin{obs}
\label{sharpk2}
There is a gap between the upper and lower bounds in \eqref{sa-320}, \eqref{sa-321}, respectively, which means that at least the lower bound is not sharp. The consequences of this fact in the study of the hitting probabilities in Section \ref{ss3-3} have been discussed in the introduction.
\end{obs}

\begin{lema}\label{l-sa-4}
Let $k=3$ and $D=B_1(0)$. Fix $\rho_0<1$. Then, there exist two positive constants 
\begin{align*}
\tilde C &= \tilde c \left(\frac{1-\rho_0}{2\rho_0}\wedge \frac{1}{19}\wedge (1-\rho_0)^4\right)^{\frac{1}{2}},\\
C &= c\frac{1}{(1-\rho_0)^{2}},
\end{align*}
with $\tilde c$ and $c$ some multiple of $\pi^{-1/2}$, such that 
for any $x_1,\, x_2\in{B}_{\rho_0}(0)$,
\beq
\label{sa-42}
\tilde C |x_1-x_2|^{\frac{1}{2}}\le \|G_D^3(x_1,\cdot)-G_D^3(x_2,\cdot)\|\leq C|x_1-x_2|^{\frac{1}{2}}.
\eeq
\end{lema}
\begin{proof}
We fix $x_1, x_2\in \bar{B}_{\rho_0}(0)$, $x_1\ne x_2$, and start by proving the upper bound. For this, we first find a bound from above for $\|L_{x_1}^3-L_{x_2}^3\|_2$, using a similar approach as for $k=2$. 
Let $x_1, x_2$ be distinct points in $B_{\rho_0}(0)$ and set
$r_{x_1x_2}= 2|x_1-x_2|$. Assume  $|x_1-x_2|>1$. Then $|y-x_1|\le r_{x_1x_2}$, for any $|y|\le 1$, and
\beqn
\|L^3_{x_1}-L^3_{x_2}\|^2 \le (8\pi^2)^{-1}[I_1(x_1) + I_2(x_2)],
\eeqn
with 
\begin{align*}
I_1(x_1) &= \int_{\{|y|\leq 1\}\cap\{|y-x_1|\leq r_{x_1x_2}\}}|x_1-y|^{-2}dy,\\
I_2(x_2) &= \int_{\{|y|\leq 1\}\cap\{|y-x_2|\leq \frac{3r_{x_1x_2}}{2}\}}|x_2-y|^{-2}dy.
\end{align*}
Applying the change of variables given by the spherical coordinates yields
\beq
\label{sa-421}
I_1(x_1) + I_2(x_2) \leq 20\pi |x_1-x_2|.
\eeq
Next, we assume that $|x_1-x_2|\le 1$. We have
\beqn
\|L^3_{x_1}-L^3_{x_2}\|^2 \leq (4\pi^2)^{-1}[I_1(x_1)+I_2(x_2))] + (8\pi^2)^{-1}I_3(x_1,x_2),
\eeqn
with
\beqn
 I_3(x_1,x_2) = \int_{\{|y|\leq 1\}\cap \{|y-x_1|> r_{x_1x_2}\}\}}\left(\frac{1}{|x_1-y|}-\frac{1}{|x_2-y|}\right)^2 dy.
 \eeqn
 A direct computation shows that $\left\vert\nabla_x\left(|\cdot-y|^{-1}\right)\right\vert = |x-y|^{-2}$. Using this fact, along with the 
mean value theorem, we obtain
\begin{align*}
I_3(x_1,x_2)& \le |x_1-x_2|^2  \int_{\{|y|\leq 1\}\cap \{|y-x_1|> r_{x_1x_2}\}\}} |x^*-y|^{-4} dy,
\end{align*}
with $x^*= x_1+\lambda(x_2-x_1)$ for some $\lambda\in(0,1)$.

On the set $\{|y-x_1| > r_{x_1x_2}\}$, 
\begin{align}
\label{sa-4210}
|x^*-y|&=|x_1-y+\lambda(x_2-x_1)|\geq |x_1-y|-\lambda|x_2-x_1|\notag\\
&> r_{x_1x_2}\left(1-\frac{\lambda}{2}\right)\geq \frac{r_{x_1x_2}}{2}.
\end{align}
Thus
\begin{align*}
I_3 (x_1,x_2)&\leq |x_1-x_2|^2  \int_{\{\frac{r_{x_1x_2}}{2}\le |x^*-y|\le 2 \}} |x^*-y|^{-4} dy
\\&\leq 4\pi |x_1-x_2|^2 \int_{\frac{r_{x_1x_2}}{2}}^2 r^{-2} dr\\
& \le 4\pi |x_1-x_2|.
\end{align*}
Thus, we have proved
\beq
\label{l}
\|L^3_{x_1}-L^3_{x_2}\|^2 \leq \frac{C}{\pi} |x_1-x_2|.
\eeq
By computing $\nabla_x S^3_{\cdot}(y)$, we see that 
\beq
\label{II}
\vert \nabla_x S^3_{\cdot}(y)\vert = 4\pi\vert S^3_x(y)\vert^2 |y| \le (4\pi)^{-1}(1-\rho_0)^{-2} |y|, \forall x\in {B}_{\rho_0}(0).
\eeq
Fix $x_1, x_2 \in \bar{B}_{\rho_0}(0)$. The preceding inequality, along with the mean value theorem yields
\begin{align}
\label{sa-43}
\|S^3_{x_1}-S^3_{x_2}\|^2&=\int_{|y|\leq 1}(S^3_{x_1}(y)-S^3_{x_2}(y))^2 dy\notag
\\&=\int_{|y|\leq 1}|\nabla_{x^*} S^3_{\cdot}(y)|^2|x_1-x_2|^2 dy\notag
\\&\leq (12\pi)^{-1}(1-\rho_0)^{-4}|x_1-x_2|^2,
\end{align}
where $x^*$ is a point lying on the interval determined by $x_1$ and $x_2$.
Together with \eqref{l}, this yields the upper bound in \eqref{sa-42}.

Let $\eta\in\left(0,\frac{1-\rho_0}{2\rho_0}\wedge\frac{1}{2}\right)$. Since $|x_1-x_2|\le 2\rho_0$,  we have $\eta |x_1-x_2|<1-\rho_0$. Let $D_{\eta} = \{y\in D: |y-x_1|<\eta |x_1-x_2|\}$. The choice of $\eta$ implies $D_\eta\subset D$, and then,
\begin{align*}
\|G_D^3(x_1,\cdot)-G_D^3(x_2,\cdot)\|^2 &\ge \|G_D^3(x_1,\cdot)-G_D^3(x_2,\cdot)\|_{\eta}^2\\
&\ge \frac{1}{2}\
\Vert L_{x_1}^3-L_{x_2}^3\Vert_\eta^2
-\Vert S_{x_1}^3-S_{x_2}^3\Vert_\eta^2,
\end{align*}
where $\Vert\cdot\Vert_{\eta}$ denotes the $L^2$-norm on $D_\eta$.

With similar computations as in \eqref{sa-43}, we see that
\begin{align}
\Vert S_{x_1}^3-S_{x_2}^3\Vert_\eta^2 & = \int_{D_\eta}(S^3_{x_1}(y)-S^3_{x_2}(y))^2 dy\notag\\
&=\int_{D_\eta}|\nabla_{x^*} S^3_{\cdot}(y)|^2|x_1-x_2|^2 dy\notag\\
&\le (4\pi)^{-2}(1-\rho_0)^{-4}|x_1-x_2|^2\int_{D_\eta} |y|^2 dy\notag\\
&\le (12\pi)^{-1} (1-\rho_0)^{-4} \eta^3 |x_1-x_2|^5. \label{sa-44}
\end{align}

Next, we prove a lower estimate for $\Vert L_{x_1}^3-L_{x_2}^3\Vert_\eta^2$. Expanding the
square of this norm yields,
\beqn
\Vert L_{x_1}^3-L_{x_2}^3\Vert_\eta^2 = (4\pi)^{-2}\left[J_1-2J_2+J_3\right],
\eeqn
with
\begin{align*}
J_1 = \int_{D_\eta} \frac{dy}{|y-x_1|^2},\ 
J_2 = \int_{D_\eta} \frac{dy}{|y-x_1|Ê|y-x_2|},\ 
J_3 = \int_{D_\eta} \frac{dy}{|y-x_2|^2}.
\end{align*}
With a change of variables to spherical coordinates, we have 
\beqn
 J_1 = 4\pi \int_0^{\eta |x_1-x_2|} dr = 4\pi\eta |x_1-x_2|.
\eeqn
To study $J_2$, we notice that since $\eta<\frac{1}{2}$, we have $|y-x_2|>\frac{|x_1-x_2|}{2}$ for any  $y\in D_\eta$. Indeed, assume that $|y-x_2|\le \frac{|x_1-x_2|}{2}$, for some $y\in D_\eta$, then by the triangular inequality,
\beqn
 |x_1-x_2| \le |x_1-y| + |y-x_2| \le \left(\eta + \frac{1}{2}\right) |x_1-x_2|,
 \eeqn
 which is a contradiction. Hence, by applying spherical coordinates, we have
 \begin{align*}
 J_2 &< \frac{2}{|x_1-x_2|} \int_{D_\eta} \frac{dy}{|y-x_1|} = \frac{8\pi}{|x_1-x_2|} \int_0^{\eta |x_1-x_2|} r dr\\
 & = 4\pi \eta^2 |x_1-x_2|.
\end{align*}
Because $\eta<\frac{1}{2}$, we see that on the set $D_\eta$, $|y-x_2| < (\eta+1)|x_1-x_2| < \frac{3}{2}|x_1-x_2|$. Thus,
\begin{align*}
J_3 > \frac{4}{9}|x_1-x_2|^{-2}\int_{D_\eta} dy = \frac{16\pi}{27}\eta^3|x_1-x_2|.
\end{align*}
The estimates on the terms $J_1, J_2, J_3$ obtained above imply,
\beq
\label{sa-441}
\Vert L_{x_1}^3-L_{x_2}^3\Vert_\eta^2 > (4\pi)^{-1} |x_1-x_2|\eta\ (1-2\eta+\frac{4}{27}\eta^2).
\eeq
Along with \eqref{sa-44}, and since $|x_1-x_2|<2$, we obtain
\begin{align}
&\|G_D^3(x_1,\cdot)-G_D^3(x_2,\cdot)\|^2\notag\\
&\qquad \ge (8\pi)^{-1} |x_1-x_2|   
[\eta\ (1-2\eta+\frac{4}{27}\eta^2)-|x_1-x_2|^4 \eta^3 (1-\rho_0)^{-4}]\notag\\
&\qquad\ge (8\pi)^{-1} |x_1-x_2| [\eta\ (1-2\eta -2^4 \eta^2 (1-\rho_0)^{-4})], \label{sa-442}
\end{align}
for any $\eta\in\left(0, \frac{1-\rho_0}{2\rho_0}\wedge\frac{1}{2}\right)$.

Let $c_1= \frac{1-\rho_0}{2\rho_0}\wedge \frac{1}{19}\wedge (1-\rho_0)^4$. The above computations show that, for any $\eta\in (c_1/2, c_1)$, 
\beqn
\|G_D^3(x_1,\cdot)-G_D^3(x_2,\cdot)\|^2 \ge(8\pi)^{-1} \frac{c_1}{38} |x_1-x_2|.
\eeqn 
This completes the proof of the lower bound in \eqref{sa-42} and of the lemma.
\hfill{$\square$}
\end{proof}



\begin{cor} 
\label{c-sa-1}
For $k=1,2,3$, we consider the setting of Lemmas \ref{l-sa-2}, \ref{l-sa-3}, \ref{l-sa-4}, respectively. Then the mapping 
$x\mapsto \sigma_x=\Vert G_D^k(x,\cdot)\Vert$, is H\"older continuous. More precisely, there exists a constant $C>0$ such
that, for any $x_1, x_2\in D$,
\beq
\label{sa-45}
\vert\sigma_{x_1}-\sigma_{x_2}\vert \le C\begin{cases}
\vert x_1-x_2\vert, & k=1,\\
\vert x_1-x_2\vert^{1-\gamma}, & k=2,\\
\vert x_1-x_2\vert^{\frac{1}{2}}, & k=3,
\end{cases}
\eeq
where $\gamma>0$ is arbitrarily small.
\end{cor}
\begin{proof} This is a consequence of the triangular inequality along with the upper bounds \eqref{sa-21}, \eqref{sa-32}, \eqref{sa-42}.
\hfill{$\square$}
\end{proof}
\begin{obs}
\label{sa-r2}
In connection with numerical approximations of the SPDE \eqref{mildPoisseq} with $d=1$ and $D= (0,1)^k$, $k=1,2,3$,
we find in \cite{GM2006} the following results.
\begin{enumerate}
\item $\sup_{x\in D}\Vert G_D^k(x,\cdot)\Vert < \ + \infty$ (\cite[Lemma 3.3 ]{GM2006}).
\item
For any $\varepsilon>0$ there exists a constant $C=C(k,\varepsilon)$ (depending on $k$ and $\varepsilon$), such that, for any $x_1, x_2\in D$,
 \begin{align*}
 \|G_D^1(x_1,\cdot)-G_D^2(x_2,\cdot)\| &\le C |x_1-x_2|^1,\\
 \|G_D^2(x_1,\cdot)-G_D^2(x_2,\cdot)\| &\le C |x_1-x_2|^{1-\varepsilon},\\
 \|G_D^3(x_1,\cdot)-G_D^3(x_2,\cdot)\| &\le C |x_1-x_2|^{\frac{1}{2}-\varepsilon}.
 \end{align*}
(see \cite[Lemma 3.4 ]{GM2006}). 
\end{enumerate}
The proof uses the development of the Green function with respect to an othonormal basis in $L^2(D)$.
\end{obs}

\medskip

\subsection{Hitting probabilities}
\label{ss3-3}

Throughout this section, we consider the following setting:
\begin{itemize}
\item Case $k=1$. $D= (0,b)$, $b>0$,  $I$ is a closed interval of $D$ satisfying $d(I,\partial D)=b_0>0$.
\item Case $k=2, 3$. $D= B_1(0)$, $I$ is a compact subset of $D$ satisfying $d(I,\partial D)=d_0>0$.
\end{itemize}

\subsubsection{Upper bounds}
\label{sss3-3-1}

In this section, $A$ denotes a non empty Borel set of $\R^d$ and we establish upper bounds of the probability
$\P\{v(I)\cap A\ne \emptyset\}$ in terms of the Hausdorff dimension of $A$. 

\begin{tma}
\label{t-s3-1}
The sets $D\subset \R^k$, $I\subset D$ and $A$ are as above. Then, there exists a constant $C$, depending on 
$D, k, d$, such that
\beq
\label{s3-3333}
\P\{v(I)\cap A\ne \emptyset\} \le C \mathcal{H}_{d-\frac{k}{\xi}}(A),
\eeq
with 
\beq
\label{parameters}
\xi= \begin{cases}
1, & k=1,\\
1-\gamma, & k=2, \\
\frac{1}{2}, & k=3,
\end{cases}
\eeq
where $\gamma>0$ is arbitrarily small.
\end{tma}
\begin{proof}
If $d\le \frac{k}{\xi}$, we have $\mathcal{H}_{d-\frac{k}{\xi}}(A)=\infty$, and \eqref{s3-3333} holds trivially. 

Let $d > \frac{k}{\xi}$. We will apply \cite[Theorem 2.6]{DSS2010} to the process $\{v(x), x\in D\}$, which relies on the following assumptions:
\begin{description}
\item{(i)} $\inf_{x\in K}\sigma_x>0$, for any compact subset $K\subset D$.
\item{(ii)} For any $\epsilon$ small enough, 
\beq
\label{s3-33}
\E\left(\int_{R_j^\epsilon} dx \int_{R_j^\epsilon} dy \left[\exp\left\{\frac{|v(x)-v(y)|}{|x-y|^\xi}\right\}\right]\right) \le C \epsilon^{\frac{2k}{\xi}},
\eeq
where $R_j^\epsilon = \Pi_{l=1}^k\left[j_l \epsilon^{\frac{1}{\xi}}, (j_l+1)\epsilon^{\frac{1}{\xi}}\right)$, $j=(j_1,\ldots,j_k)$, $j_1, \ldots,j_k\in \mathbb{Z}$, and $R_j^\epsilon \cap I\neq \emptyset$.
\end{description}
Property (i) has already been proved. Hence, we put our efforts in proving (ii). 

By the isometry property of the stochastic integral and Lemmas \ref{l-sa-2}, \ref{l-sa-3}, \ref{l-sa-4} (see the upper bounds in \eqref{sa-21}, \eqref{sa-320}, \eqref{sa-42}, respectively), we have
\begin{align}
\label{s3-38}
 \E(\vert v(x)-v(y)\vert^2) & = \Vert G^k_D(x,\cdot) - G^k_D(y,\cdot)\Vert^2\notag\\
 & \le C |x-y|^{2\xi},
 \end{align}
 with $\delta$ given in \eqref{parameters}.
 
This implies 
 \beqn
 \frac{|v^i(x)-v^i(y)|}{|x-y|^\xi} \le C \frac{|v^i(x)-v^i(y)|}{[E(|v^i(x)-v^i(y)|^2)]^{\frac{1}{2}}},
 \eeqn
 $i=1,\ldots,d$. 
 
 Let $\Lambda_{x,y}$ be the covariance matrix of the Gaussian random vector $v(x)-v(y)$, that is, 
 \beqn
 \Lambda_{x,y}=\left(\left[E(\vert v^i(x)-v^i(y)\vert^2)\right]^{\frac{1}{2}}\delta_i^j\right)_{1\le i,j\le d},
 \eeqn
 $\delta_i^j$ being the Kronecker symbol.
 
 The law of the random vector $Z:= \Lambda_{x,y}^{-1} [v(x)-v(y)]$ is $N_d(0, {\text{Id}})$. Consequently,
 \begin{align*}
\E\left(\int_{R_j^\epsilon} dx \int_{R_j^\epsilon} dy \left[\exp\left\{\frac{|v(x)-v(y)|}{|x-y|^\xi}\right\}\right]\right)
& \le C \int_{R_j^\epsilon} dx \int_{R_j^\epsilon} dy \ \E \left[\exp |Z|\right]\\
& \le C \epsilon^{\frac{2k}{\xi}}.
\end{align*}
Hence, \eqref{s3-33} holds. \hfill{$\square$}
\end{proof}

\subsubsection{Lower bounds}
\label{sss3-3-2}

In this section, we consider the dimensions $k=1, 3$. We refer to the introductory section for remarks relative to the dimension $k=2$.
We have the following result.

\begin{tma}
\label{t-s3-2}
Let $k=1, 3$. Fix $N>0$ and a Borel set $A\subset [-N,N]$. There exists a positive constant $c$ depending on the set $D$ and the parameters $d, N$, such that
\beq
\label{s3-39}
\P\{v(I)\cap A \ne \emptyset\} \ge c\ {\text{Cap}}_{d-\frac{k}{\xi}}(A),
\eeq
with $\xi$ given in \eqref{parameters}.
\end{tma}
\smallskip

Before giving the proof of this theorem, we observe that from \eqref{sa-21}, \eqref{sa-42} and the definition of the pseudometric $\delta$ given in \eqref{metric}, we have
 \beq
\label{delta}
c|x_1-x_2|^\xi\le\delta(x_1,x_2)\le C |x_1-x_2|^\xi,
\eeq
for some positive constants $c, C$, and for any $x_1,x_2\in I$,
where 
\beqn
\xi=\begin{cases}
1, & k=1\\
\frac{1}{2}, & k=3.
\end{cases}
\eeqn  
\medskip

\noindent{\em Proof of Theorem \ref{t-s3-2}.}  We apply \cite[Theorem 2.1]{DSS2010} to the stochastic process $v$ defined in \eqref{s3-11}. This accounts to check the following statements.
\begin{enumerate}
\item For any $x\in I$, the density function $z\mapsto p_{v(x)}(z)$ is continuous and bounded. Moreover, $p_{v(x)}(z)>0$ for any $z$ on a compact set of $\R^d$.
\item  For any $x_1, x_2 \in I$, $x_1\ne x_2$, the joint density of $(v(x_1), v(x_2))$, $p_{x_1,x_2}$, exists and satisfies this property:

Fix $M>0$. There exists $\gamma, \alpha>0$ such that $\frac{2}{\alpha}(\gamma-k)= d-\frac{k}{\xi}$ ($\xi$ defined in \eqref{parameters}), and 
\beq
\label{s3-53}
p_{x_1,x_2}(z_1,z_2)
\le \frac{C}{|x_1-x_2|^{\gamma}} \exp\left(-\frac{c|z_1-z_2|^2}{|x_1-x_2|^\alpha}\right),
\eeq
for any $z_1,z_2\in [-M,M]^d$, where $C, c$ are positive constants independent of $x_1, x_2$.
\end{enumerate}

Property 1 follows from \eqref{s3-2}-\eqref{s3-500}. 
Along with Lemma \ref{l-s4-1}, we infer the existence of the joint density $p_{x_1,x_2}$.
\smallskip

\noindent{\em Case k=1}
\smallskip

We fix $i\in\{1,\ldots,d\}$, and denote by
 $p_{x_1,x_2}^i(z_1,z_2)$,  $p_{x_2|x_1}^i(z_2|z_1)$, $p_{x_1}^i(z_1)$ the joint density of $(v^i_{x_1}, v^i_{x_2})$ at $(z_1,z_2)$, the conditional density of $v^i_{x_2}$ at point $z_2$ given $v^i_{x_1}= z_1$, and the marginal density of $v^i_{x_1}$ at $z_1$, respectively. Then, by linear regression,  
 \begin{align*}
p_{x_1,x_2}^i(z_1,z_2) & = p_{x_2|x_1}^i(z_2|z_1)p_{x_1}^i(z_1)\\
&= \frac{1}{\sqrt{2\pi}\tau_{x_1,x_2}} \exp\left(-\frac{|z_2-m_{x_1,x_2}z_1|^2}{2\tau_{x_1,x_2}^2}\right)\notag\\
& \quad \times \frac{1}{\sqrt{2\pi}\sigma_{x_1}} \exp\left(-\frac{|z_1^2|}{2\sigma_{x_1}^2}\right)
\end{align*}
where $m_{x_1,x_2}$, $\tau_{x_1,x_2}^2$ denote the conditional mean and variance, respectively (the definitions are recalled in Lemmas \ref{l-s4-2} and \ref{l-s4-3}).

As in the proof of \cite[Proposition 3.1]{DSS2010}, by simple algebraic manipulations, we obtain
\begin{align}
\label{k1.1}
p_{x_1,x_2}^i(z_1,z_2)& 
\le \frac{1}{2\pi\sigma_{x_1}\tau_{x_1,x_2}}
\exp\left(-\frac{|z_1-z_2|^2}{4\tau_{x_1,x_2}^2}\right)\notag\\
& \quad \times \exp\left(\frac{|z_1|^2|1-m_{x_1,x_2}|^2}{2\tau_{x_1,x_2}^2}
  \right)\exp\left( -\frac{|z_1|^2}{2\sigma_{x_1}^2}\right).
\end{align}
In order to get \eqref{s3-53} (with $\frac{2}{\alpha}(\gamma-1)= d-1$) from \eqref{k1.1}, we will use \eqref{s3-3}, \eqref{s3-46}, and prove that
\beq
\label{s3-23}
c_1|x-y|^2 \le 1-\rho^2_{xy}\le  c_2 |x-y|^2, 
\eeq
for any $x,y\in I$, where $c_1$, $c_2$ are positive constants. 

The upper bound in \eqref{s3-23} follows from \eqref{s3-3} and \eqref{s3-47}, and is valid in any dimension $k$. A complete proof of \eqref{s3-23} in dimension $k=1$ can be done as follows. 

By definition, 
\beqn
1-\rho^2_{xy}= \frac{(\sigma_x\sigma_y-\sigma_{xy})(\sigma_x\sigma_y+\sigma_{xy})}{\sigma_x^2\sigma_y^2}.
\eeqn
Based on the expression \eqref{GD1}, with direct computations we obtain
\beqn
\sigma_{xy} = \frac{xy}{6b}(2b^2-3bx-3by+x^2+y^2) + \frac{xy(x\wedge y)}{2} - \frac{(x\wedge y)^3}{6},
\eeqn
which yields
\beqn
\sigma_x\sigma_y-\sigma_{xy} = \frac{(x\wedge b)(b-(x\vee y))(x-y)^2}{6b}.
\eeqn
From the three equations above and \eqref{s3-3}, we deduce \eqref{s3-23}.

Going back to \eqref{k1.1} and because of the independence of the components $v^i$, the estimates \eqref{s3-23} imply the inequality \eqref{s3-53} with $\gamma=d$, $\alpha=2$.
This proves the lower bound \eqref{s3-39} when $k=1$.

\noindent{\em Case $k=3$}
\smallskip

By Lemma \ref{l-s4-4}, proved later on in this section, and $\eqref{delta}$, we obtain
\beq
\label{s3-52}
\left\vert\sigma_{x_1}^2 - \sigma_{x_2}^2\right\vert \le C \delta(x_1,x_2)^{1+\eta},
\eeq
with some $\eta>0$. This fact, together with (a) and (b) in Section \ref{S3-2} yields that 
the Gaussian stochastic process $\{v(x), x\in I\}$ satisfies the hypotheses of \cite[Proposition 3.1]{DSS2010}. Thus, according to that Proposition, if we fix $M>0$,  for any $x_1, x_2 \in I$, the joint density of $(v(x_1), v(x_2))$ satisfies 
\beqn
p_{x_1,x_2}(z_1,z_2) \le \frac{C}{(\delta(x_1,x_2))^d} \exp\left(-\frac{c|z_1-z_2|^2}{(\delta(x_1,x_2))^2}\right),
\eeqn
where $C,c$ are positive constants independent of $x_1,x_2$ and $z_1,z_2\in[-M,M]^d$.
Because of \eqref{delta}, the right-hand side of the above inequality is bounded by 
\beq
\label{s3-530}
\frac{C}{|x_1-x_2|^{\frac{d}{2}}} \exp\left(-\frac{c|z_1-z_2|^2}{|x_1-x_2|}\right).
\eeq
Hence, Property 2. above holds with $\gamma:= \frac{d}{2}$ and $\alpha:= 1$, which according to the conclusion of \cite[Theorem 2.1]{DSS2010} yields  \eqref{s3-39} for $k=3$.
\hfill{$\square$}

\begin{obs}
\label{s3-01}
Theorems \ref{t-s3-1} and \ref{t-s3-2} have the following consequences.

Let $k=1,3$. By the definition of the Hausdorff dimension ${\text{dim}}_{\text H}$ (see e.g. \cite {Ka1993} or \cite{K2002}) and Frostman's theorem, we have
\begin{align*}
{\text{dim}}_{\text H}(A) &< d-\frac{k}{\xi} \Longrightarrow \{v(I) \cap A = \emptyset\}\  a.s.\\
{\text{dim}}_{\text H}(A) &> d- \frac{k}{\xi} \Longrightarrow \P\{v(I) \cap A \neq \emptyset\} >0,
\end{align*}
with $\xi$ defined in \eqref{parameters}.

Moreover, we see that if $d > \frac{k}{\xi}$, points ($A= \{y_0\}$) are polar for $v$, while there are non polar if $d <\frac{k}{\xi}$.

If $d=\frac{k}{\xi}$ ({\em critical dimension}), the results obtained so far for the hitting probabilities are not informative. For example, if $A= \{y_0\}$, it says
\beqn
0\le \P\{\exists x\in I:  v(I)= y_0\}  \le 1.
\eeqn
We refer to \cite{DMX2015} for a method to characterise polarity of points for Gaussian random fields at critical dimensions with applications to the heat and wave stochastic equations.
\end{obs}
\medskip


In the particular case $k=2$, Theorem \ref{t-s3-1} implies that if $d>2$, points are polar for $v$.

We close this section with an auxiliary result used in the proof of Theorem \ref{t-s3-2}.

\begin{lema}
\label{l-s4-4}
Let $k=3$. Fix $\rho_0\in(0,1)$. There exists $C>0$ such that for all $x_1, x_2\in B_{\rho_0}(0)$,
\beq
\label{s3-49}
\left\vert\sigma_{x_1}^2 - \sigma_{x_2}^2\right\vert \le C \left\vert x_1 - x_2\right\vert^{1-\zeta},
\eeq
with $\zeta>0$ arbitrarily small.
\end{lema}
\begin{proof}
Let $r_{x_1,x_2}= 2\left\vert x_1 - x_2\right\vert$ and let
\begin{align*}
D_1 & = D\cap \{|y-x_1|\le r_{x_1,x_2}\},\\
D_2 & = D\cap \{|y-x_1| > r_{x_1,x_2}\}.
\end{align*}
By definition,
\begin{align*}
\sigma_{x_1}^2 - \sigma_{x_2}^2 & = \int_{D_1} \left[|G_D^3(x_1,y)|^2 - |G_D^3(x_2,y)|^2 \right] dy\\
& + \int_{D_2} \left[|G_D^3(x_1,y)|^2 - |G_D^3(x_2,y)|^2 \right] dy.
\end{align*}
Consider the expression of $G_D^3$ given in \eqref{sa-31}. As observed in \cite[p. 19]{GT2001},  
for all $x,y\in \bar D$, $G_D^3(x,y)\ge 0$ (notice that in the notation of that reference, $G_D^3 (x,y) = - G(x,y)$). Hence, with the notation  
\eqref{sa-311}, we have $S_x^3(y) \le L_x^3(y)$, for any $x,y\in \bar D$. Therefore, 
\begin{align}
\label{sa-50}
 &\left\vert\int_{D_1} \left[|G_D^3(x_1,y)|^2 - |G_D^3(x_2,y)|^2 \right] dy\right\vert  
 \le 2 \int_{D_1} \left[(L_{x_1}^3(y))^2 + L_{x_2}^3(y))^2\right] dy\notag \\
 &\quad\quad\le C \left(\int_{D_1} |x_1-y|^{-2} dy + \int_{D_1} |x_2-y|^{-2} dy\right)\notag \\
 &\quad\quad\le C \left\vert x_1 - x_2\right\vert,
 \end{align} 
 where the last inequality follows from \eqref{sa-421}.
 
 Our next aim is to find an upper bound for 
 \beqn
 \left\vert\int_{D_2} \left[|G_D^3(x_1,y)|^2 - |G_D^3(x_2,y)|^2 \right] dy\right\vert.
 \eeqn
  For this, we apply the mean value theorem to the function
$x\mapsto (G_D^3(x,y))^2$ and obtain
\beqn
(G_D^3(x_1,y))^2 - (G_D^3(x_2,y))^2 = 2 G_D^3(x^*,y) \nabla_x G_D^3(x^*,y)(x_1-x_2),
\eeqn
where $x^*= \lambda x_1 + (1-\lambda) x_2$, for some $\lambda\in (0,1)$. This yields
\begin{align*}
&\left\vert\int_{D_2} \left[|G_D^3(x_1,y)|^2 - |G_D^3(x_2,y)|^2 \right] dy\right\vert
\le C \left\vert x_1 - x_2\right\vert\\
& \quad\times \int_{D_2} \vert G_D^3(x^*,y)\vert \vert  \nabla_x G_D^3(x^*,y)\vert dy.
\end{align*}
For all $x\in D$ and $\gamma\in(0, 3)$, the integral $\int_D |G_D^3(x,y)|^\gamma dy$ is finite.
Apply H\"older's inequality with $\gamma\in(0, 3)$, $\bar\gamma= \frac{\gamma}{\gamma-1}$ (
observe that $\bar\gamma>\frac{3}{2}$). We obtain
\begin{align*}
Z(x^*)&: = \int_{D_2} \vert G_D^3(x^*,y)\vert \vert  \nabla_x G_D^3(x^*,y)\vert dy\\
&\le \left(\int_{D_2} |G_D^3(x^*,y)|^\gamma dy\right)^{\frac{1}{\gamma}}
\left(\int_{D_2} \vert  \nabla_x G_D^3(x^*,y)\vert^{\bar\gamma} dy\right)^{\frac{1}{\bar\gamma}}\\
&\le C \left(\int_{D_2} \vert  \nabla_x G_D^3(x^*,y)\vert^{\bar\gamma} dy\right)^{\frac{1}{\bar\gamma}}.
\end{align*}
We pursue the proof with the study of 
\beqn
Y(x^*):=\int_{D_2} \vert  \nabla_x G_D^3(x^*,y)\vert^{\bar\gamma} dy.
\eeqn
Using the expression \eqref{sa-3111}, we see that
\beqn
Y(x^*) \le C \left( \int_{D_2} \vert  \nabla_x L_{x^*}^3(y)\vert^{\bar\gamma} dy
+ \int_{D_2} \vert  \nabla_x S_{x^*}^3(y)\vert^{\bar\gamma} dy\right).
\eeqn
Since 
\beqn
 \vert \nabla_x L_{x}^3(y)\vert = \vert \nabla_x (|x-y|^{-1})\vert = |x-y|^{-2},
 \eeqn
 and on the set $D_2$, we have $|x^* -y|\ge |x_1-x_2|$ (see \eqref{sa-4210}), we obtain
 \begin{align*}
 \int_{D_2} \vert\nabla_x L_{x^*}^3(y)\vert^{\bar\gamma} dy
 &=  \int_{D_2} |x^*-y|^{-2\bar\gamma} dy\\
 &\le C \int_{|x_1-x_2|}^3 r^{2-2\bar\gamma} dr\\ 
 &= \frac{C}{2\bar\gamma-3}\left(\frac{1}{|x_1-x_2|^{2\bar\gamma-3}} - \frac{1}{3^{2\bar\gamma-3}}\right)\\
&\le C |x_1-x_2|^{3-2\bar\gamma},
 \end{align*}
 since $3-2\bar\gamma<0$.
 By using \eqref{II}, we have
 \begin{align*}
 \int_{D_2} \vert  \nabla_x S_{x^*}^3(y)\vert^{\bar\gamma} dy \le C \int_{B_2(0)} |y|^{\bar\gamma} dy < \infty.
 \end{align*}
 Consequenly, we have proved that
 \begin{align}
 \label{sa-51}
 \left\vert\int_{D_2} \left[|G_D^3(x_1,y)|^2 - |G_D^3(x_2,y)|^2 \right] dy\right\vert
 &\le C \left(|x_1-x_2| + |x_1-x_2|^{1+\frac{3-2\bar\gamma}{\bar\gamma}}\right)\notag\\
 &\le C |x_1-x_2|^{1+\frac{3-2\bar\gamma}{\bar\gamma}},
 \end{align}
 because $3-2\bar\gamma<0$.
 
 The upper bound \eqref{sa-51}, along with \eqref{sa-50} implies
 \beqn
 \left\vert\sigma_{x_1}^2 - \sigma_{x_2}^2\right\vert \le C \left\vert x_1 - x_2\right\vert^{1+\frac{3-2\bar\gamma}{\bar\gamma}}.
 \eeqn
 By choosing $\gamma\in(0,3)$ arbitrarily close to $3$, we have $\bar\gamma>\frac{3}{2}$ and arbitrarily close to $\frac{3}{2}$.  Thus, the exponent $1+\frac{3-2\bar\gamma}{\bar\gamma}= \frac{3}{\bar\gamma}-1$ will be less than, but arbitrarily close to $1$. Hence, there exists $\eta>0$ such that \eqref{s3-49} holds. 
\hfill{$\square$}
\end{proof}

\medskip

\noindent{\bf Acknowledgements}. 
The authors express their gratitude to Joan Verdera, for useful guidance concerning the results of Section \ref{sa}.


\end{document}